\documentclass[reqno]{amsart}
\usepackage{textcomp}

\usepackage[dvipsnames]{xcolor}
\usepackage{amsmath}
\usepackage{amsthm}
\usepackage{hyperref}
\usepackage{amsfonts,graphics,amsthm,amsfonts,amscd,latexsym}
\usepackage{tikz-cd}
\usepackage{epsfig}
\usepackage{flafter}
\usepackage{longtable}
\usepackage{mathtools}
\usepackage{comment}
\usepackage{stmaryrd}

\usepackage{mathabx,epsfig}

\hypersetup{
    colorlinks=true,    
    linkcolor=blue,          
    citecolor=blue,      
    filecolor=blue,      
    urlcolor=blue           
}
\usepackage{tikz}
\usetikzlibrary{graphs,positioning,arrows,shapes.misc,decorations.pathmorphing}

\tikzset{
    >=stealth,
    every picture/.style={thick},
    graphs/every graph/.style={empty nodes},
}

\tikzstyle{vertex}=[
    draw,
    circle,
    fill=black,
    inner sep=1pt,
    minimum width=5pt,
]
\usepackage[position=top]{subfig}
\usepackage{amssymb}
\usepackage{color}

\calclayout

\usetikzlibrary{decorations.pathmorphing}
\tikzstyle{printersafe}=[decoration={snake,amplitude=0pt}]

\newcommand{\vol}{\operatorname{vol}}

\newcommand{\pp}{\mathbb{P}}

\newcommand{\qq}{\mathbb{Q}}
\newcommand{\zz}{\mathbb{Z}}

\newcommand{\cc}{\mathbb{C}}

\def\O#1.{\mathcal {O}_{#1}}   
\def\pr #1.{\mathbb P^{#1}}    
\def\af #1.{\mathbb A^{#1}}   
\def\ses#1.#2.#3.{0\to #1\to #2\to #3 \to 0} 
\def\xrar#1.{\xrightarrow{#1}}   
\def\K#1.{K_{#1}}      
\def\bA#1.{\mathbf{A}_{#1}}   
\def\bM#1.{\mathbf{M}_{#1}}    
\def\bL#1.{\mathbf{L}_{#1}}    
\def\bB#1.{\mathbf{B}_{#1}}    
\def\bK#1.{\mathbf{K}_{#1}}   
\def\subs#1.{_{#1}}     
\def\sups#1.{^{#1}}

\newcommand{\lif}{\upharpoonleft \kern-0.35em}

\usepackage{tikz}
\usetikzlibrary{matrix,arrows,decorations.pathmorphing}

  \newtheorem{theorem}{Theorem}[section]
  
  \newtheorem{proposition}[theorem]{Proposition}
  \newtheorem{corollary}[theorem]{Corollary}
  
  \newtheorem{conjecture}[theorem]{Conjecture}

  \newtheorem{definition}[theorem]{Definition}
  \newtheorem{example}[theorem]{Example}
  
  \newtheorem{problem}[theorem]{Problem}

\theoremstyle{remark}

\numberwithin{equation}{section}

\usepackage[all]{xy}

\begin{document}

\thanks{J.M. was partially supported by NSF research grant DMS-2443425.}

\title[Cluster type varieties]
{Cluster type varieties}

\author[J.~Moraga]{Joaqu\'in Moraga}
\address{UCLA Mathematics Department, Box 951555, Los Angeles, CA 90095-1555, USA
}
\email{jmoraga@math.ucla.edu}

\subjclass[2020]{Primary 14M25, 14E25;
Secondary  14B05, 14E30, 14E05.}
\keywords{Toric geometry, cluster geometry, cluster type varieties.}

\begin{abstract}
Cluster type varieties are compactifications of algebraic tori on which the volume form has no zeros. 
These form a natural class of varieties that generalizes both toric varieties and cluster varieties. 
The aim of this article is to introduce the reader to the concept
of cluster type varieties and explain some recent results towards the understanding of these varieties. 
At the same time, we will pose some problems for further research.
\end{abstract} 

\maketitle

\setcounter{tocdepth}{1} 
\tableofcontents

\section{Introduction} 

This survey aims to introduce the reader to the concept of cluster type variety.
To define cluster type varieties, we start with an algebraic torus $\mathbb{G}_m^n$ with variables $x_1,\dots,x_n$.
We can define a volume form 
\[
\Omega_n := \frac{dx_1\wedge \dots \wedge dx_n}{x_1\cdots x_n} 
\]
on $\mathbb{G}_m^n$.
The volume form $\Omega_n$ is a top differential form in $\mathbb{G}_m^n$ without zeros or poles. 
Let $j\colon \mathbb{G}_m^n \dashrightarrow X$ be an embedding in codimension one, i.e., a birational map which is an embedding up to a closed subset of codimension two in $\mathbb{G}_m^n$. We can think of $j_*\Omega_n$
as a rational top differential form on $X$. 
The variety $X$ is said to be {\em cluster type} if $j_*\Omega_n$ has no zeros on $X$ (see Definition~\ref{def:cluster-type}). 
The poles of the volume form $j_*\Omega_n$ defines a boundary divisor $B$ on $X$ which we call the {\em cluster type boundary}
and the pair $(X,B)$ is called a {\em cluster type pair}. 

The concept of cluster type varieties generalizes toric varieties and cluster varieties (the spectrum of cluster algebras).
These are probably the two most important classes of cluster type varieties. 
The definition of cluster type varieties is recent; an early pioneer definition was given by
Corti, Filip, and Petracci in~\cite[Definition 3.4]{CFP20}.
The current definition was given by Enwright, Figueroa, and the author in~\cite[Section 1.4]{EFM24}. 
The definition given in~\cite[Section 1.4]{EFM24} is more birational in nature, however, it is equivalent to the definition given in this introduction.  
Since their definition, cluster type varieties have sparked substantial research and interest. 
Some of the works include~\cite{CFP20, corti2023cluster,Mor24b,EFM24,LMV24,JM24,AdSFM24,Fil25,MY24,ELY25,MZ25}. In this survey, we explain these works and related research, and propose new directions for future investigation. At the same time, we try to provide some historical background on standard notions in algebraic geometry. 
We also present new theorems and examples that can help understand this theory.

For the rest of the introduction, we provide a brief overview of the topics and results covered in each section of this survey.
In Section~\ref{sec:toric-geometry}, we recall notions in toric geometry. 
We will review the relation between smooth projective toric varieties
and smooth integral polytopes. 
We will discuss some problems related to asymptotics in the classification of smooth/Gorenstein toric Fano varieties. 
In Section~\ref{sec:cluster-geometry}, we will discuss the geometric side of cluster algebras and define the notion of cluster type varieties. For the rest of that section, we will mostly focus on surfaces and prove that the notion of cluster type pairs of dimension two $(X,B)$ coincides with that of Looijenga pairs whenever $X$ is smooth (see Proposition~\ref{prop:Looijenga-pairs-are-cluster-type}). We will also explain that the understanding of cluster type pairs of dimension two is much harder when the underlying surface has Gorenstein singularities (see Example~\ref{ex:ct-sing-surf}). In Section~\ref{sec:sing-theory} and Section~\ref{sec: bir-geom}, we make an interlude to discuss concepts related to singularity theory and birational geometry. 
These sections will review concepts about singularities of pairs and crepant birational pairs. The content of these two sections is not central to the survey, but they provide useful language for explaining the geometry of higher-dimensional cluster type pairs. At the same time, we prove some new results, in Section~\ref{sec: bir-geom}, we show that a log rational pair admits a crepant blow-up which is cluster type. We refer the reader to Definition~\ref{def:log-rational} for the concept of log rational pairs. In Section~\ref{sec:dim-3}, we will discuss the known results about three-dimensional cluster type pairs. 
We begin this section by discussing the log rationality of three-dimensional Calabi-Yau pairs. 
We recall some results due to Ducat~\cite{Duc24} in the case of pairs $(\pp^3,B)$ and some results due to Loginov, Vasilkov, and the author for pairs $(X,B)$ where $X$ is a smooth Fano threefold~\cite{LMV24}. 
Then, we dive into detecting which pairs of the form $(\pp^3,B)$ are cluster type, or equivalently, which quartic surfaces $B\subset \pp^3$ admit algebraic tori in their complement $\pp^3\setminus B$.
In Section~\ref{sec:complexity}, we discuss how the complexity and the birational complexity relate to notions of toric geometry and cluster geometry. 
We will review theorems stating that Calabi--Yau pairs of complexity zero are toric~\cite{BMSZ18}, while smooth Calabi--Yau pairs of complexity one are cluster type~\cite{ELY25}. We pose some questions regarding Calabi--Yau pairs of higher complexity.
In Section~\ref{sec:finite-actions}, we discuss finite actions on Fano varieties.
First, we recall the Jordan property for finite actions on $n$-dimensional Fano varieties. 
Then, we connect the topic of finite actions on Fano varieties to that of cluster type varieties. 
More precisely, we review a theorem stating that Fano varieties with {\em large} finite actions must be cluster type (see Theorem~\ref{thm:finite-actions}). 
Then, we propose a geometric version of the Jordan property for finite subgroups of the Cremona (see Conjecture~\ref{conj:geom-jordan}). 
In Section~\ref{sec:ct-in-families}, we study the property of being cluster type in families of Fano varieties. 
We discuss a theorem, due to Lena Ji and the author, stating that the cluster type property is constructible in families of smooth Fano varieties. Then we explore theorems on degenerations of cluster type pairs. 
Finally, in Section~\ref{sec:ct-and-cox-rings}, we study 
the Cox rings of cluster type varieties. At the same time, we discuss some results about Fano compactifications of the spectrum of cluster algebras.\\

\textbf{Acknowledgments}
These notes are based on a talk and a lecture series by the author at the Bootcamp of the Summer Research Institute in Algebraic Geometry 2025 at Colorado State University. The author would like to thank the organizers of the SRI 2025.
At the same time, the author would like to thank the participants in their lectures during the Bootcamp for their comments and fruitful discussions: 
Felipe Castellano-Macias,
Baidehi Chattopadhyay,
Lauren Cranton Heller, 
Eric Jovinelly, 
Hyunsuk Kim, 
Chih-Kuang Lee, 
Roktim Mascharak,
Niklas M\"uller,
Supravat Sarkar, and
Juan Pablo Z\'u\~niga.

\section{Toric geometry}\label{sec:toric-geometry}

Toric geometry emerged in the 1970s with the study of torus embeddings. Some early foundational figures in toric geometry 
are Demazure, Kempf, Knudsen, Mumford, Saint-Donat, Miyake, and Oda 
that linked lattices and polyhedral cones to algebraic varieties (see, e.g.,~\cite{KKMS-D73}).
In the early ages, Danilov and Khovanskii made some important contributions related to cohomological aspects of toric varieties~\cite{DK86}.
Some crucial concepts, such as the orbit-cone correspondence, the language of fans, and connections with convex geometry, were achieved
by the aforementioned mathematicians.
Danilov's work introduced the terminology {\em toric varieties}
through a translation due to Miles Reid. 
In the 1980s, it became a prolific topic
leading to some breakthroughs as Stanley's proof of McMullen's conjecture using toric ideals~\cite{Sta80}.
In the 1990s, Sturmfels popularized the concept of toric ideals 
and made important connections with commutative algebra
and Gr\"obner basis~\cite{Stu97}.
In the modern era, toric geometry is a fundamental tool in algebra.
In various fields of algebra, it's been a crucial tool for producing explicit constructions and offers computational advantages. 
Toric geometry has important applications in algebraic geometry, commutative algebra, birational geometry, mirror symmetry, and related areas.
Nowadays, the classic references for toric geometry are the works
of Cox, Little, Schenck~\cite{CLS11}, Fulton~\cite{Ful93}, and Danilov~\cite{Dan78}. 

There are many equivalent definitions of toric varieties.
The standard definition of a toric variety is the following. 

\begin{definition}
{\em 
A $n$-dimensional algebraic variety $X$ is said to be {\em toric} 
if there is an open embedding $\mathbb{G}_m^n \hookrightarrow X$ 
such that the action of $\mathbb{G}_m^n$ on itself 
extends to a group action on $X$.
}
\end{definition} 

The way to relate toric varieties with combinatorial objects is simple. We look at all the possible one-parameter subgroups 
$\phi_{\vec{a}} \colon \mathbb{G}_m \rightarrow \mathbb{G}_m^n$. These are given by 
\[
t \mapsto (t^{a_1},t^{a_2},\dots,t^{a_n}),
\]
where $\vec{a}=(a_1,\dots,a_n)\in \mathbb{Z}^n$. 
Then, for each $\vec{a}\in \zz^n$, we can consider the limit 
$x_{\vec{a}}:=\lim_{t\rightarrow 0} \phi_{\vec{a}}(t)$, if it exists. 
This gives an equivalence relation of $\zz^n$; we can identify the points $\vec{a}$ with the same limit $x_{\vec{a}}$. This leads to a subdivision of $\qq^n$ into polyhedral cones. This subdivision is not a partition, as polyhedral cones may intersect along faces. This leads to the concept of {\em polyhedral fans}\footnote{Polyhedral fans, or fans of polyhedral cones, are often simply called {\em fans} in the theory.}. For instance, the fan of $\mathbb{G}_m^n$ is just the origin in $\qq^n$, 
the fan of $\mathbb{A}^n$ is the cone ${\rm span}(e_1,\dots,e_n)$ in $\mathbb{Q}^n$, and the fan of $\mathbb{P}^n$ is the set of all cones spanned by the vectors $\{e_1,\dots,e_n,-(e_1+\dots+e_n)\}$ in $\qq^n$. Polyhedral fans in $\qq^n$ are in correspondence with $n$-dimensional toric varieties. 
To make this correspondence into a bijection, one needs to 
identify fans up to the natural action of ${\rm GL}(n,\zz)$ on $\qq^n$ and identify toric varieties up to toric isomorphism. For more details, we refer the reader to~\cite[Section 2.3]{CLS11}.

There is another natural way to construct polyhedral fans in $\qq^n$
which comes from convex geometry.
Given a convex integral polytope $P\subset \qq^n$, we can consider the space 
of linear functionals $(\qq^n)^*$ restricted to $P$.
There is an equivalence relation on these linear functionals; 
two functionals belong to the same equivalence class if they minimize along 
the same face of $P$. 
Linear functionals on the same equivalence class form polyhedral cones in $(\qq^n)^*$. This construction leads to a polyhedral fan in 
$\qq^n$ (the dual space) which is known as the {\em normal fan} of the convex polytope $P$. 
There is a natural observation about normal fans, i.e., they cover the whole dual space $\qq^n$. We call a fan {\em complete} when the union of all the cones of the fan is $\qq^n$. 
For instance, the fan of $\mathbb{P}^n$ is complete while
the fan of $\mathbb{A}^n$ is not complete. 

One can ask whether the fan of a toric variety is a normal fan. It turns out that the fan associated to a toric variety $X$ is the normal fan of a polyope $P$ if and only if $X$ is a projective variety~\cite[Theorem 2.3.1]{CLS11}. One can observe that many polytopes have the same normal fan; this happens, for instance, for all the polytopes $nP$, with $n\in \mathbb{Z}_{\geq 1}$, given an initial polytope $P$. 
One can prove that a polytope $P$ defines an ample line bundle $\mathcal{L}_P$ on the projective toric variety $X(N_P)$. 
Here, $N_P$ is the normal fan of the polytope $P$. 
Thus, a polytope leads to a polarized toric variety~\cite[Proposition 6.1.4]{CLS11}.
Below, we may write $X(P)$ for the projective toric variety associated to the polytope $P$.
For instance, the standard simplex in $\qq^n$, with vertices 
$0,e_1,\dots,e_n$, leads to $\mathcal{O}(1)$ in $\mathbb{P}^n$.
The $f$-dimensional faces of $P$ turn out to correspond to $f$-dimensional tori in $X(N_P)$ (see~\cite[Theorem 3.2.6]{CLS11}). Every toric variety can be written as the disjoint union of algebraic tori. This is made explicit via the orbit-cone correspondence.
It is worth mentioning that there are complete toric varieties which are not projective (see, e.g.,~\cite[Exercise 3.4.1]{Ful93}). 
Many properties of the toric variety $X(P)$ can be read from the corresponding polytope $P$. For instance, $X(P)$ is a smooth variety if and only if $P$ is a smooth polytope, i.e., the lattice generators at each vertex form a basis of $\zz^n$ (see~\cite[Theorem 2.4.3]{CLS11}). 
The toric variety $X(P)$ is $\qq$-factorial if and only if $P$ is a simplicial polytope, i.e., the lattice generators at each vertex form a basis of $\qq^n$ (see~\cite[Proposition 4.2.7]{CLS11}).

In algebraic geometry, {\em Fano varieties} play a central role.
These varieties are often used in classification problems and singularity theory. A normal variety $X$ is said to be {\em Fano}
if the anti-canonical divisor $-K_X$ is an ample divisor, i.e., 
$-mK_X$ is very ample for some $m\gg 0$. 
A normal variety $X$ is said to be {\em Gorenstein} if $K_X$ is a Cartier divisor. 
A projective toric variety $X(P)$ is a Gorenstein Fano variety if and only if $P$ is a reflexive polytope, i.e., if $P$ contains the origin and its polar dual is an integral polytope~\cite{KN13}. 
A projective toric variety $X(P)$ is a smooth Fano variety if and only
if $P$ is a smooth reflexive polytope. 
The previous equivalences show that classifying smooth and Gorenstein toric Fano varieties of a fixed dimension is a problem that can be tackled combinatorially.
This can be seen as the combinatorial portion of the classification of smooth or Gorenstein Fano varieties. 
The smooth reflexive polytopes of dimension two, up to isomorphism, are the following:
\begin{center}
\includegraphics[scale=0.5]{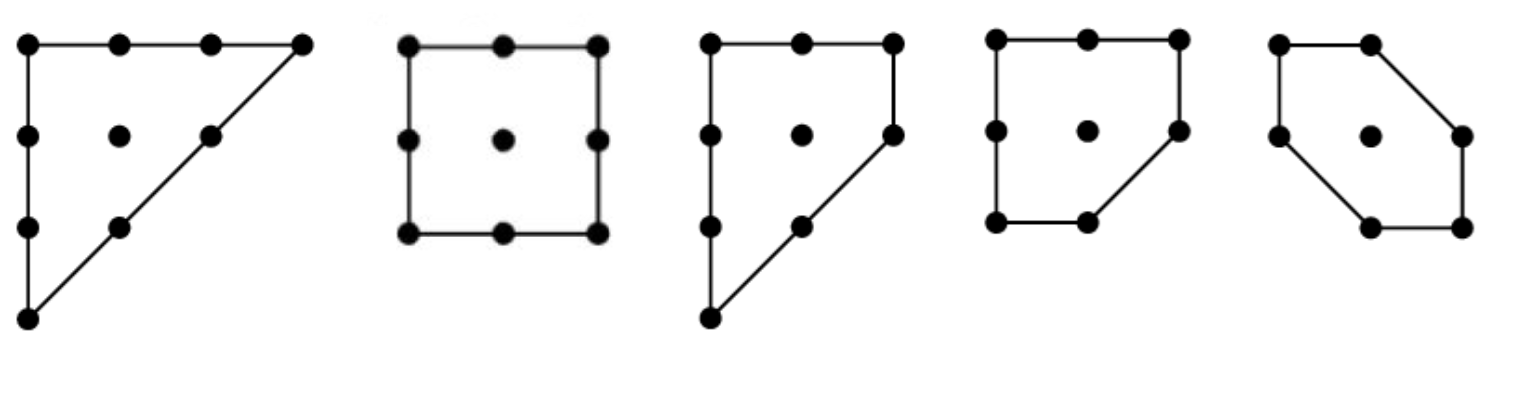}
\end{center}
These corresponds to $\pp^2,\pp^1\times \pp^1, {\rm Bl}_p(\pp^2), 
{\rm Bl}_{p,q}(\pp^2)$, and ${\rm Bl}_{p,q,r}(\pp^2)$, respectively. 
Here, we write ${\rm Bl}_{p,q,r}(\pp^2)$ for the blow-up of $\pp^2$ at three points $p,q$, and $r\in \pp^2$. More details on blow-ups of toric varieties will be explained in Section~\ref{sec: bir-geom}.

The classification of reflexive polytopes is a broad topic.
In dimension two, there are $16$ reflexive polytopes, of which $5$ are smooth.
The classifications in dimension three and four were
achieved by Kreuzer and Skarke in~\cite{KS98,KS00}.
The number of reflexive polytopes grows quickly with the dimension; there are $4,319$ in dimension $3$ and $473,800,776$ in dimension $4$.
A classification in dimension at least $5$ seems unfeasible.
The situation is much nicer for smooth reflexive polytopes. 
In the smooth setting, there are $5$ smooth reflexive polygons, $18$ smooth reflexive polytopes of dimension three, and $124$ smooth reflexive polytopes of dimension four (see~\cite{Bat99}).
There is an algorithm due to $\varnothing$bro that allows to produce the list of all
$n$-dimensional smooth reflexive polytopes~\cite{Obr07}. 
$\varnothing$bro classified such varieties up to dimension $8$ 
and Lorenz and Paffenholz classified such varieties up to dimension $9$. These can be found in the \href{https://polymake.org/polytopes/paffenholz/www/fano.html}{Smooth Reflexive Lattice Polytopes Database}.
Nevertheless, we do not have precise asymptotic bounds for the number
of $n$-dimensional smooth reflexive polytopes. 
A simple argument using products shows that the number of $n$-dimensional smooth reflexive polytopes is asymptotically at least $5.86^n$.  
There are no known exponential upper bounds for this number. This leads to the following problem.

\begin{problem}\label{problem:counting-smooth-reflexive}
Let $\mathcal{R}_n$ be the number of smooth reflexive lattice polytopes. 
Find two constants $c_1,c_2>0$ such that we have $c_1^n < \mathcal{R}_n < c_2^n$.
\end{problem}

On the contrary, the number of $n$-dimensional reflexive polytopes seems to increase doubly exponentially in $n$ (see~\cite{Kre08}). 
At any rate, we have a plethora of tools for classifying reflexive polytopes with prescribed conditions.
Henceforth, the classification of smooth/Gorenstein toric Fano varieties is a {\em decidable} problem.

One may wonder what happens for singular toric Fano varieties. 
For instance, a weighted projective space $\pp(1,1,n)$ with $n\geq 3$ is toric and Fano; however, it is not Gorenstein. 
For small $n$, we have that $\pp^2\simeq \pp(1,1,1)$ is smooth and $\pp(1,1,2)$ is a Gorenstein toric Fano surface.
In this direction, the previous sequence of examples shows that the number of such toric Fano varieties is not finite. 
However, we can get some finiteness results by imposing some constraints on the singularities of the toric Fano variety. 
More precisely, there is an invariant called {\em minimal log discrepancy} (mld for short), denoted ${\rm mld}(X)$, which associates a positive rational number or $-\infty$
to any $\qq$-Gorenstein variety $X$.
Philosophically, the larger ${\rm mld}(X)$ is, the nicer the singularities of $X$ are. Indeed, we expect ${\rm mld}(X)\leq \dim X$ for any $\qq$-Gorenstein variety with equality only happening when $X$ is smooth (see~\cite{Amb06b}).
The details of the minimal log discrepancy function will be explained
in Section~\ref{sec:sing-theory}. We will compute ${\rm mld}(\pp(1,1,n))=2n^{-1}$ in Example~\ref{ex:toric-mld}. 
A $\mathbb{Q}$-Gorenstein variety $X$ with ${\rm mld}(X)\geq \epsilon$ will be said to be {\em $\epsilon$-log canonical} ($\epsilon$-lc for short). 
The toric case of the BAB conjecture predicted that for $\epsilon>0$, there are only finitely many $n$-dimensional $e$-lc toric Fano varieties. 
This conjecture was proved by Borisov and Borisov in~\cite{BB92}. 
Thus, in principle, one can classify these varieties. 
It is unclear how the asymptotic behavior of the number of toric Fano varieties depends on the mld. 
Finding precise asymptotics for the number of these varieties is an interesting question, but $\varnothing$bro's work
points out that it is a challenging problem. Even some computational evidence towards new bounds would be of interest. 
The previous discussion motivates the following problem.

\begin{problem}\label{prob:e-lc-Fano}
Fix $\epsilon>0$. 
Let $\mathcal{T}_{n,\epsilon}$ be the number of $\epsilon$-lc toric Fano varieties of dimension $n$. Find two functions $f_1(n,\epsilon)$ and $f_2(n,\epsilon)$ such that 
$f_1(n,\epsilon) < \mathcal{T}_{n,\epsilon} < f_2(n,\epsilon)$. Describe how $\mathcal{T}_{n,\epsilon}$ depends on $\epsilon$.
\end{problem}

Now, we turn to explore some different characterizations of toric varieties 
and show how weakening these characterizations leads to new interesting classes of algebraic varieties.
The following theorem characterizes projective toric varieties in terms of their automorphism group.

\begin{theorem}\label{thm:rank-charact}
Let $X$ be a $n$-dimensional projective variety. The variety $X$ is toric if and only if 
${\rm Aut}^0(X)$ has reductive rank $n$. 
\end{theorem} 

Note that the previous theorem only works in the projective setting. 
In the affine setting, the group ${\rm Aut}^0(X)$ may not be algebraic, so the previous theorem does not make sense.
Theorem can be thought of as the definition of projective toric varieties. 
We can weaken the definition by imposing that ${\rm Aut}^0(X)$ has rank $r\geq 0$. 
This leads to the concept of {\em $\mathbb{T}$-varieties}. 
A {\em $\mathbb{T}$-variety} is a normal variety admitting an effective torus action. 
The $\mathbb{T}$ in the notation of $\mathbb{T}$-variety represents the acting algebraic torus.
The $\mathbb{T}$-complexity of a $\mathbb{T}$-variety $X$ equals $\dim(X)$ minus the dimension of the acting torus. 
Hence, a toric variety is nothing but a $\mathbb{T}$-variety of $\mathbb{T}$-complexity zero.
The theory of $\mathbb{T}$-varieties was initiated by Altmann, Hausen, and Suss in the seminal work~\cite{AH06,AHS08}. 
In these papers, the authors introduce combinatorial language to describe $\mathbb{T}$-varieties in terms
of their normalized Chow quotient.
They define the concept of {\em polyhedral divisors} and {\em divisorial fans}. 
A polyhedral divisor is a formal sum $\sum_{i=1}^k P_iD_i$ where the $D_i$'s are Weil divisors on a normal variety $Y$
while the $P_i$'s are polyhedral in a common $\qq$-vector space $N_\qq$. 
The main theorem of~\cite{AH06} states that a $\mathbb{T}$-variety $X$ can be encrypted by a polyhedral divisor 
on its normalized Chow quotient $Y$. 
The theory of $\mathbb{T}$-varieties is very useful to study varieties with torus actions.
It is particularly strong when the normalized Chow quotient is a curve; 
the so-called complexity one $\mathbb{T}$-varieties.
Complexity one $\mathbb{T}$-varieties play a central role in the theory as they are still mostly combinatorial objects.
Some extra difficulties arise when the normalized Chow quotient is a surface or a higher-dimensional variety. 
However, the theory of $\mathbb{T}$-varieties can still provide some interesting results in this direction. 
During the 2010s, there was a substantial effort to generalize theorems from toric geometry to $\mathbb{T}$-varieties. 
Of course, many properties of toric varieties, such as rationality, are not satisfied by general $\mathbb{T}$-varieties.
Nevertheless, a $\mathbb{T}$-variety is rational if and only if its normalized Chow quotient is rational. 
Hence, the main focus was always to find the right conditions on the normalized Chow quotient and polyhedral 
divisors to recover good properties of toric varieties. 
In this direction, many important results were obtained regarding; Cox rings~\cite{AP12}, singularities~\cite{LS13}, cohomology groups~\cite{LLM20}, Hodge theory~\cite{AVL14} , differential forms~\cite{PS11}, fundamental groups~\cite{LLM19}, etc. 

To conclude this section, we turn to a volume-theoretic characterization of projective toric varieties. 
In what follows, we write $\Omega_n$ for the volume form of the $n$-dimensional algebraic torus $\mathbb{G}_m^n$.
More precisely, if $\mathbb{G}_m^n={\rm Spec}(K[x_1^{\pm},\dots,x_n^{\pm}])$, then 
\[
\Omega_n:=\frac{dx_1\wedge \dots \wedge dx_n}{x_1\cdots x_n}.
\] 
The following theorem was proved in~\cite[Theorem 1.10]{EFM24} using the language of birational geometry that we explain below. 

\begin{theorem}\label{thm:vol-charact}
Let $X$ be a $n$-dimensional variety. 
The variety $X$ is toric if and only if there is an open embedding 
$\mathbb{G}_m^n\hookrightarrow X$ such that $\Omega_n$ has poles
at every prime component of $X\setminus \mathbb{G}_m^n$.
\end{theorem} 

Note again that the projectivity assumption in Theorem~\ref{thm:vol-charact} is necessary. 
Indeed, consider the variety $U:=\mathbb{G}_m\times \mathbb{A}^1-\{(1,0)\}$.
It is clear that we have a torus embedding $\mathbb{G}_m^2\hookrightarrow U$
and the volume form $\Omega_2$ has no zeros on $U$. 
However, looking at the fundamental group of $U$, which is isomorphic to $\zz$, 
we conclude that the only toric surfaces it could be isomorphic to are 
$\mathbb{G}_m\times \mathbb{A}^1$ or $\mathbb{G}_m\times \mathbb{P}^1$ (see, e.g.,~\cite[Theorem 12.1.5]{CLS11}).
The first case can be discarded by looking at invertible regular functions,
while the second case can be discarded by looking at higher cohomologies of the structure sheaf.
Theorem~\ref{thm:vol-charact} can be thought of as the definition of projective toric varieties. 

Our approach is similar to that of $\mathbb{T}$-varieties, 
in the sense that we focus on a class of algebraic varieties that emerge when relaxing the conditions on a characterization of toric varieties.
One can naturally weaken the hypothesis of Theorem~\ref{thm:vol-charact} and impose that $\Omega_n$ has no {\em zeros} on $X\setminus \mathbb{G}_m^n$, 
instead of asking that $\Omega_n$ has poles at every component.
This leads to the concept of cluster type varieties: 
\begin{definition}\label{def:cluster-type}
{\em A $n$-dimensional algebraic variety $X$ is said to be of {\em cluster type} if there is an embedding in codimension one $\mathbb{G}_m^n\dashrightarrow X$ such that $\Omega_n$ has no zeros on $X$.
The effective divisor $B={\rm Poles}(\Omega_n)$ on $X$ is called a {\em cluster type boundary}. The pair $(X,B)$ is called a {\em cluster type pair}.
}
\end{definition}
In the previous definition, an embedding in codimension one 
$\mathbb{G}_m^n\dashrightarrow X$ means that there exists a closed subset $Z\subset \mathbb{G}_m^n$ of codimension at least two for which there exists an embedding 
$\mathbb{G}_m^n\setminus Z \hookrightarrow X$. 
Note that every toric variety is a cluster type variety
and the reduced toric boundary is a particular instance
of a cluster type boundary.
Thus, a fixed variety may have multiple cluster type boundaries. Definition~\ref{def:cluster-type} appeared first in~\cite{EFM24}, however, Corti already gave a very similar definition with the extra assumption on the singularities of the poles of $\Omega_n$ in $X$ (see~\cite[Definition 2]{corti2023cluster}). We call these varieties {\em cluster type} as they share some geometric features with spectra of cluster algebras. Indeed, if $(X,B)$ is cluster type, then the variety $X\setminus B$ resembles the spectrum of a cluster algebra. The word {\em type} in algebraic geometry is often used to emphasize that some phenomena happen up to some boundary divisor. 
Examples of these are the definitions of {\em Fano type}, {\em klt type}, and {\em Calabi--Yau type}, which reflect the definition of Fano, klt, and Calabi--Yau, respectively, up to adding a boundary divisor $B$ to the variety $X$ (see Definition~\ref{def:ft}, Definition~\ref{def:klt-type}, and Definition~\ref{def:cy-type} below). We will see that if $(X,B)$ is of cluster type, then $X\setminus B$ is a cluster of algebraic tori up to a closed subset of codimension two. 
In the next section, we briefly review the history of cluster geometry, i.e., cluster algebras and their spectra. 
Then, we show some basic properties of cluster type varieties and focus on the case of surfaces.

\section{Cluster geometry}
\label{sec:cluster-geometry}

The origins of cluster geometry lie in the work of Lusztig, Fomin, and Zelevinsky. 
In the mid 1990s Lusztig discovered some connections between total positivity and quantized enveloping algebras~\cite{Lus90}.
In~\cite{FZ99}, Fomin and Zelevinsky discovered some combinatorial structures on the coordinate ring of double Bruhat cells that would later become cluster algebras.
The formal definition of cluster algebras was introduced in Fomin and Zelevinsky's seminal papers~\cite{FZ02i,FZ03ii,FZ05iii,FZ07iv}.
The initial aim of this work was to understand dual canonical bases and total positivity in semisimple Lie groups.
They introduced the concept of mutations; an innovative process to generate {\em cluster variables} from an initial {\em seed}. One surprising aspect was that all cluster variables turned out to be Laurent polynomials in the initial seed; this is known as the Laurent phenomenon.
Fock and Goncharov brought geometric ideas to the field
introducing cluster varieties as a way to study higher Teichm\"uler theory and mirror symmetry~\cite{Fock2009cluster}. 
Some further developments of cluster algebras are the categorification of cluster algebras~\cite{BMRRT06}, the geometric realization via triangulation of surfaces~\cite{FST08,FT18}, and the classification of finite type cluster algebras~\cite{FZ03ii}.
In~\cite{GHK15birational}, Gross, Hacking, and Keel prove that cluster algebras are Cox rings of suitable volume-preserving blow-ups of toric varieties. Using this, they provide some counter-examples to conjectures due to Fock and Goncharov regarding dual canonical bases.
In~\cite{GHKK18}, Gross, Hacking, Keel, and Kontsevich prove that cluster algebras have natural bases given by theta divisors, proving a corrected version of the Fock-Goncharov dual basis conjecture. Nowadays, cluster algebras have become a standard tool across several fields of mathematics. 

The definition of cluster algebras is outside the scope of this survey. We refer the reader to~\cite[Definition 2.3]{FZ02i} for the formal definition of a cluster algebra.
However, we do focus on the geometric aspects of such algebras. One of the main geometric features of cluster algebras is the following. Given a $n$-dimensional cluster algebra $R$, the corresponding {\em cluster variety} $U={\rm Spec}(R)$ is covered by copies of $\mathbb{G}_m^n$, up to a closed subset $Z\subset U$ of codimension at least two. 
Moreover, given a birational map $\phi\colon \mathbb{G}_m^n\dashrightarrow \mathbb{G}_m^n$ between two such torus embeddings, the map $\phi$ preserves the volume form, i.e., $\phi^*\Omega_n=\Omega_n$ up to scalar. The following is proved in~\cite[Theorem 1.3.(3)]{EFM24} and~\cite[Theorem 23]{corti2023cluster}.

\begin{theorem}\label{thm:structural-thm-ct}
Let $X$ be a cluster type variety and $B\subset X$ be its cluster type boundary. Then, up to a closed subset of codimension at least two, the open subvariety $U:=X\setminus B$ is covered by images of embeddings in codimension one $\mathbb{G}_m^n \dashrightarrow X$. 
Furthermore, for any two embeddings in codimension one 
$j_1\colon \mathbb{G}_m^n\dashrightarrow X$ and 
$j_2\colon \mathbb{G}_m^n\dashrightarrow X$ the birational map $j_2^{-1}\colon j_1\colon \mathbb{G}_m^n \dashrightarrow \mathbb{G}_m^n$ preserves the volume form $\Omega_n$.
\end{theorem} 

Thus, from the geometric perspective cluster varieties
and cluster type varieties are quite similar. 
In both cases, we have a cluster of algebraic tori, and the birational transition maps between them are volume-preserving.
The main difference is the formulas defining these birational transition maps.
Indeed, in the case of cluster varieties, the birational maps 
$\mathbb{G}_m^n \dashrightarrow \mathbb{G}_m^n$ can be decomposed into simple {\em mutations} of the form 
\begin{equation}\label{eq:mutation} 
(x_1,\dots,x_n) \mapsto \left(
x_1,\dots,\frac{p(x_1,\dots,\hat{x_i},\dots,x_n)}{x_i},\dots,x_n 
\right)
\end{equation} 
where $p$ is a suitable irreducible binomial whose exponents are determined by a skew-symmetric matrix (see~\cite[Proposition 4.3]{FZ02i} for the details). 
In the case of cluster type varieties the birational maps between torus charts can be decomposed into similar birational maps as in~\eqref{eq:mutation}. However, in the cluster type case, the polynomial $p$ may be arbitrary, possibly reducible.
From this perspective, the class of cluster type varieties is much closer to the notion of Laurent phenomenon algebras introduced by Lam and Pylyavskyy~\cite{LP16}.
There are examples in which $(X,B)$ is a cluster type pair and $\mathcal{O}(X\setminus B)$ is neither a cluster algebra nor a Laurent phenomenon algebra (see~\cite[Example 2.11]{EFMS25}). 
We will discuss the structure of these rings in Section~\ref{sec:ct-and-cox-rings}.

Of course, the only normal projective cluster type curve is the projective line $\mathbb{P}^1$. 
For the rest of this section, we will focus on trying to understand the geometry of cluster type surfaces. 
We only focus on normal projective varieties. 

First, assume that $X$ is a smooth projective surface. 
First, let's assume that $X$ is of cluster type. 
Then, we have an embedding in codimension one 
$\mathbb{G}_m^2 \dashrightarrow X$. 
In the case of surfaces, this means that we have an embedding 
$\mathbb{G}_m^2 \setminus \{p_1,\dots,p_r\}\hookrightarrow X$. 
Let $\Omega_2$ be the volume form on $\mathbb{G}_m^2$ 
which can also be regarded as a rational $2$-differential form on $X$. As $\Omega_2$ does not have zeros on $X$ and 
$X$ is a smooth projective surface, we conclude that $\Omega_2$ defines a non-trivial divisor of poles $B$. 
Thus, we conclude that $B\in |-K_X|$.
This argument does not really use the dimension of $X$; every 
cluster type variety with mild singularities satisfies $|-K_X|\neq \emptyset$.
We try to understand the geometry of the curve $B$. 
Note that we have a commutative diagram 
\[
\xymatrix{
 & \mathbb{G}_m^2 \ar@{-->}[ld]_-{j_1}\ar@{-->}[rd]^-{j_2} & \\ 
\pp^2 \ar@{-->}[rr]^-{\pi} & & X 
}
\]
where $j_1$ and $j_2$ are embeddings in codimension one and $\pi$ is a birational map. Let $p\colon Y\rightarrow X$ be a sequence of blow-ups of $X$ for which there is a projective birational morphism 
$q\colon Y\rightarrow \pp^2$, i.e., $p$ and $q$ resolve the birational map $\pi$. 
Write $p^*(K_X+B)=K_Y+B_Y+E-F$ where $B_Y$ is the strict transform of $B$ in $Y$, and $E,F\geq 0$ are $p$-exceptional divisors. We may assume that $E$ and $F$ have no common components. 
Note that $K_X+B\sim 0$, so $K_Y+B_Y+E-F\sim 0$.
Let $q_*(K_Y+B_Y+E-F)=K_{\pp^2}+q_*B_Y+q_*E-q_*F$.
Observe that the divisor $q_*B_Y+q_*E-q_*F$ must be supported on $\pp^2\setminus j_1(\mathbb{G}_m^2)$.
By the negativity lemma~\cite[Proposition 3.39]{KM98}, we conclude that 
$q_*B_Y+q_*E-q_*F\sim -K_{\pp^2}$ and 
$q_*B_Y+q_*E-q_*F$ is supported along the coordinate hyperplanes of $\pp^2$. 
There is just one effective divisor that satisfies this property;
namely the reduced sum of the coordinate hyperplanes
$H_0+H_1+H_2$. Thus, we conclude that 
\begin{equation}\label{eq:boundary-agree}
q_*B_Y+q_*E-q_*F = H_0+H_1+H_2.
\end{equation}
From equality~\eqref{eq:boundary-agree}, we conclude that $F$ must be contracted in $\pp^2$. 
In other words, the blow-ups that extracted $F$ were not really necessary. Thus, we may replace $Y$ by another birational model in which no such curve in $F$ is extracted. By doing so, we may assume that $F=0$. 
Similarly, we conclude that $E$ is a reduced divisor. 

We turn to analyze the birational map from the perspective of $\pp^2$.
Let $K_Y+\Gamma_Y=q^*(K_{\pp^2}+H_0+H_1+H_2)$. 
Applying the negativity lemma again, we see that 
$\Gamma_Y=B_Y+E$. 
Thus, the birational morphism $Y\rightarrow \pp^2$ only consists of blow-ups along the boundary divisor $H_0+H_1+H_2$. 
From this, we conclude that $B_Y+E$ is a cycle of smooth rational curves. 
The birational map $Y\rightarrow X$ is a sequence of contractions of smooth rational curves. 
Each curve contracted by $Y\rightarrow X$ satisfies $K_Y\cdot C=-1$ and so there are two possibilities for $C$:
\begin{enumerate}
\item $C$ is a smooth rational curve that intersects $B_Y+E$ transversally ; or 
\item $C$ is a component of $B_Y+E$.
\end{enumerate} 
In the first case, contracting $C$ does not change the isomorphism class of $B_Y+E$. In the second case, the cycle of smooth rational curves loses one curve or it becomes a nodal rational curve.
We conclude that $B$ is either a cycle of smooth rational curves or a nodal rational curve. 
We conclude the following proposition.

\begin{proposition}\label{prop:nodal-curve}
Let $X$ be a smooth projective surface. 
If $X$ is of cluster type, then the poles of the volume form of the algebraic torus is an anti-canonical nodal curve.
\end{proposition}

Note that the fact that $B\in |-K_X|$ is nodal implies the two cases mentioned above. Whenever we say that a curve is nodal, we mean that it has at least one node, otherwise we will simply say that the curve is smooth. 
From Proposition~\ref{prop:nodal-curve}, we conclude that a cluster type surface is a rational surface $X$ admitting a nodal curve $B$ on its anti-canonical system $|-K_X|$. 
This is precisely the definition of Looijenga pairs:

\begin{definition}\label{def:Looijenga-pair}
{\em 
A pair $(X,B)$ of a smooth rational projective surface $X$
and a curve $B\in |-K_X|$ is said to be a {\em Looijenga pair} if $B$ is nodal. 
}
\end{definition} 

Looijenga pairs were introduced to study the smoothability of cusps singularities~\cite{Loo81}. Friedman and Scattone further explored the relation between Looijenga pairs and degenerations of K3 surfaces~\cite{FS86}. Nowadays, Looijenga pairs are an important case of study in mirror symmetry~\cite{GHK15}, moduli theory~\cite{AAB24}, and cluster algebras~\cite{GHK15birational}.
In~\cite{GHK15}, Gross, Hacking, and Keel study the mirrors of Looijenga pairs. They introduce the following notions of {\em toric blow-ups} and {\em strong toric models}\footnote{In~\cite{GHK15}, these are called toric models. However, we will use that denomination for some weaker notion in higher dimensions.} for Looijenga pairs: 

\begin{definition}
{\em
Let $(X,B)$ be a Looijenga pair. A {\em toric blow-up} of $(X,B)$ is a projective birational morphism 
$\pi\colon X'\rightarrow X$ such that if $B'$ is the reduced scheme structure of $\pi^{-1}(B)$, then $(X',B')$ is again a Looijenga pair. 
}
\end{definition} 

In the setting of the previous definition, we may say that 
$(X',B')\rightarrow (X,B)$ is a toric blow-up.
It is easy to see that blowing up the nodes of $B$ gives toric blow-ups. Indeed, a toric blow-up of a Looijenga pair is nothing but a sequence of blow-ups of nodes. 
These are called toric blow-ups as they are isomorphic to 
toric blow-ups $T\rightarrow \mathbb{A}^2$ locally near the node. 

\begin{definition}\label{def:strong-toric-model}
{\em
A {\em strong toric model} of a Looijenga pair $(X,B)$ is a projective birational morphism $\pi\colon X\rightarrow X'$ such that:
\begin{enumerate}
\item $X'$ is a toric surface; and 
\item if $B'$ is a torus invariant divisor of $X'$, then
$B\rightarrow B'$ is an isomorphism. 
\end{enumerate}
}
\end{definition}

In the previous definition, we may say that $(X,B)\rightarrow (X',B')$ is a strong toric model for the Looijenga pair. Note that not all Looijenga pairs admit strong toric models in the sense of Definition~\ref{def:strong-toric-model}. 
For instance, consider the pair $(\pp^2,C_3)$ where $C_3$ is a nodal cubic. The only projective birational morphism $\pp^2\rightarrow T$ is the identity and $C_3$ is not torus invariant. 
However, a beautiful result due to Gross, Hacking, and Keel shows that every Looijenga pair admits a strong toric model after possibly performing a sequence of toric blow-ups (see~\cite[Proposition 1.3]{GHK15}). 

\begin{proposition}\label{prop:Looijenga-pairs-are-cluster-type}
Let $(X,B)$ be a Looijenga pair. Then, there exists a toric blow-up $p\colon (Y,B_Y)\rightarrow (X,B)$ and a strong toric model
$q\colon (Y,B_Y)\rightarrow (T,B_T)$. 
\end{proposition} 

Note that $X\setminus B \simeq Y\setminus B_Y$ by construction.
On the other hand, we have $T\setminus B_T\simeq \mathbb{G}_m^2$ and hence we have an embedding 
$\mathbb{G}_m^2 \hookrightarrow Y\setminus B_Y$. 
Therefore, there is an embedding $\mathbb{G}_m^2\hookrightarrow X\setminus B$. In this case, the volume form $\Omega_2$ of the torus does not vanish on $X$ and $B$ is the set of poles of $\Omega_2$. Indeed, similarly to the argument sketched in the previous page, we can check that 
\[
p^*(K_X+B)=K_Y+B_Y \text{ and } q^*(K_T+B_T)=K_Y+B_Y.
\]
Putting Proposition~\ref{prop:nodal-curve}
and Proposition~\ref{prop:Looijenga-pairs-are-cluster-type} together, we conclude the following theorem.

\begin{theorem}\label{thm:ct-vs-Loo}
Let $X$ be a smooth projective surface.
A pair $(X,B)$ is of cluster type if and only if $(X,B)$ is a Looijenga pair. 
\end{theorem} 

In other words, a smooth projective surface $X$ is of cluster type if and only if the linear system $|-K_X|$ admits a nodal curve. Thus, deciding whether a smooth projective surface is of cluster type or not is a rather simple problem. 
As a consequence, one can completely understand whether smooth Fano surfaces are of cluster type or not. 
Recall that smooth Fano surfaces are also known as {\em del Pezzo} surfaces and their degree is defined to be the self-intersection of the anti-canonical divisor. 
The following theorem was first observed in~\cite[Theorem 2.1 and Remark 2.2]{ALP23}.

\begin{theorem} 
Let $X$ be a smooth del Pezzo surface. Then, the following statements hold:
\begin{enumerate} 
\item If $X$ has degree at least two, then it is cluster type. 
\item If $X$ is a general surface of degree one, then it is cluster type.
\end{enumerate}
\end{theorem} 

There are examples of smooth del Pezzo surfaces $X$ of degree one where all the elements of $|-K_X|$ are either smooth or have cusp singularities. Therefore, these del Pezzo surfaces of degree one are not cluster type. The understanding of cluster type {\em structures} given by embeddings in codimension one $\mathbb{G}_m^n\dashrightarrow X$ is then equivalent to the study of boundary divisors $B\in |-K_X|$ for which 
there is an embedding in codimension one $\mathbb{G}_m^n\dashrightarrow X\setminus B$. 
One can now easily list all the possible cluster type boundaries in the projective plane $\mathbb{P}^2$: 
\begin{itemize}
\item three transversal lines;
\item a quadric and a transversal line; and 
\item a nodal cubic.
\end{itemize} 
Note that there is moduli on the cluster type boundaries in $\pp^2$. Another interesting fact is that the number of embeddings in codimension one $\mathbb{G}_m^n\dashrightarrow X$ to a cluster type variety tends to be infinite, even if we fix some cluster type boundary $B\subset X$. 
In~\cite[Theorem 1.10]{EFM24}, the authors prove that a cluster type pair $(X,B)$
admits a unique embedding in codimension one $\mathbb{G}_m^n \dashrightarrow X\setminus B$ if and only if $(X,B)$ is a toric pair. 
In the case of the cluster type pair $(\pp^2,Q+L)$, where $Q$ is a quadric and $L$ is a transversal line, we can find precisely two tori. 
We have two embeddings 
\[
j_1\colon \mathbb{G}_m^2 \hookrightarrow \pp^2 \setminus {(Q+L)}
\quad 
\text{ and }
\quad 
j_2\colon \mathbb{G}_m^2 \hookrightarrow \pp^2 \setminus {(Q+L)}.
\]
These embeddings can be explained as follows. 
Let $p_1$ and $p_2$ be the two points of intersection of $Q$ and $L$. Let $L_1$ and $L_2$ be the two lines in $\pp^2$ which are tangent to $Q$ at $p_1$ and $p_2$, respectively.
Then, the image of $j_1$ is $\pp^2\setminus (Q+L+L_1)$
and the image of $j_2$ is $\pp^2\setminus (Q+L+L_2)$.
Indeed, after some toric blow-ups of $(\pp^2,Q+L)$ the strict transforms of the lines $L_1$ and $L_2$ become $(-1)$-curves that can be contracted to some strong toric models. These toric models are in correspondence with tori in the complement of the cluster type boundary. The case $(\pp^2,C_3)$ where $C_3$ is a nodal cubic is much more interesting. Indeed, there are infinitely many embeddings $\mathbb{G}_m^2 \hookrightarrow \pp^2\setminus C_3$.

Understanding cluster type surfaces gets 
much more difficult in the singular case. 
In~\cite[Proposition 5.1]{EFM24}, Enwright, Figueroa, and the author prove that a cluster type surface only has $A_n$ singularities outside the cluster type boundary, i.e., singularities locally isomorphic to $\{(x,y,z)\mid x^2+y^2+z^{n+1}=0\}$. The analysis of the cluster type boundary $B\in |-K_X|$ above is still valid in the orbifold case, i.e., if $X$ is an orbifold and $(X,B)$ is a cluster type pair, then $B$ is a nodal curve. However, Example~\ref{ex:ct-sing-surf} below shows that Theorem~\ref{thm:ct-vs-Loo} does not hold for singular surfaces. For Example~\ref{ex:ct-sing-surf}, we will use the following proposition regarding fundamental groups of cluster type pairs.

\begin{proposition}\label{prop:fundamental-ct}
Let $(X,B)$ be a cluster type pair. 
Then, the fundamental group $\pi_1(X^{\rm sm}\setminus B)$ is abelian of rank at most $\dim X$.
\end{proposition}

\begin{proof}
We have an embedding in codimension one $\mathbb{G}_m^n\dashrightarrow X^{\rm sm}\setminus B$. 
Thus, there exists a closed subvariety $Z\subset \mathbb{G}_m^n$ of codimension at least two for which 
there is an embedding $\mathbb{G}_m^n\setminus Z \hookrightarrow X^{\rm sm}\setminus B$. 
Note that $\pi_1(\mathbb{G}_m^n\setminus Z)\simeq 
\pi_1(\mathbb{G}_m^n)\simeq \zz^{\dim X}$.
The previous isomorphism holds as $\pi_1(\cc^n\setminus \{0\})\simeq \{1\}$ for $n\geq 2$.
Therefore, by~\cite[Theorem 12.1.5]{CLS11}, we have a surjective homomorphism 
$\zz^{\rm dim X}\rightarrow \pi_1(X^{\rm sm}\setminus B)$.
\end{proof} 

\begin{example}\label{ex:ct-sing-surf}
{\em 
We show an example of a pair $(X,B)$ where $X$ is a surface with only $A_1$ singularities, $B$ is an anti-canonical nodal curve, and the pair $(X,B)$ is not cluster type.
Consider the pair $(\pp^1\times \pp^1,Q+F)$ where $Q$ is a general divisor in the system $|\mathcal{O}_{\pp^1\times \pp^1}(2,1)|$ and $F$ is a general fiber of the projection onto the second component. 
Let $\pi\colon \pp^1\times \pp^1 \rightarrow \pp^1$ be such projection. 
Let $p_0$ and $p_1$ be the two points of $\pp^1$ over which $Q$ ramifies and let $p_\infty$ be the image of $F$ in $\pp^1$.
Let $p_1\colon X_1\rightarrow \pp^1\times \pp^1$ be the blow-up of the points $(F_{p_0}+F_{p_1})\cap Q$. 
One can check $p_1^*(K_{\pp^1\times \pp^1}+Q+F)=K_{X_1}+Q_1+F_1$, 
where $Q_1$ and $F_1$ are the strict transforms of $Q$ and $F$, respectively. 
Let $E_1$ and $E_2$ be the two exceptional divisors of $p_1$.
Let $p_2\colon X_2\rightarrow X_1$ be the blow-up of 
$E_1\cap Q_1$ and $E_2\cap Q_2$. 
We can check $p_2^*(K_{X_1}+Q_1+F_1)=K_{X_2}+Q_2+F_2$ where 
$Q_2$ and $F_2$ are the strict transforms of $Q_1$ and $F_1$, respectively.
Note that $X_2$ has four $(-2)$-curves; the strict transforms of $F_{p_0},F_{p_1},E_1$, and $E_2$.
Let $q\colon X_2\rightarrow X$ be the contraction of these four $(-2)$-curves. 
Let $Q_X$ and $F_X$ be the push-forward of $Q_2$ and $F_2$ to $Y$, respectively. Set $B:=Q_X+F_X$.
Note that $(X,B)$ is a pair where $X$ is a surface with only $A_1$ singularities and $B$ is an anti-canonical nodal curve.

We argue that $(X,B)$ is not a cluster type pair. 
Note that we have a fibration $f\colon X\rightarrow \pp^1$ induced by $\pi$. 
By~\cite[Lemma 2.18]{GLM23}, we have a surjective homomorphism of fundamental groups 
\[
\pi_1(X^{\rm sm}\setminus B) \rightarrow \pi_1\left(\pp^1,\frac{1}{2}(p_0+p_1)+p_\infty \right)
\]
where the group on the right is the orbifold fundamental group of the pair (see, e.g.,~\cite[Definition 2.3]{GLM23}). 
The fundamental group $\pi_1(\pp^1,\frac{1}{2}(p_0+p_1)+p_\infty)$ is isomorphic to the infinite dihedral group $\zz_2\ast \zz_2$, so it is not abelian.
By Proposition~\ref{prop:fundamental-ct}, we conclude that the pair $(X,B)$ is not cluster type. 
}
\end{example}

Although the setting of singular surfaces is more difficult, there are some partial results towards the understanding of cluster type surfaces. 
Gorenstein del Pezzo surfaces of Picard rank one were classified by Miyanishi and Zhang~\cite{MZ88}. 
These surfaces are precisely the Fano surfaces with $A,D,E$-singularities and Picard rank one. 
These form a natural class of surfaces to understand from the cluster type perspective.
In~\cite{MY24}, Y\'a\~nez and the author proved the following theorem by studying the anti-canonical linear systems of such surfaces.

\begin{theorem}~\label{thm:ct-sing-surf}
Let $X$ be a Gorenstein del Pezzo surface of Picard rank one. 
The surface $X$ is of cluster type if and only if the following two conditions hold:
\begin{enumerate}
\item[(i)] $X$ only has $A_n$ singularities; and 
\item[(ii)] either $\vol(X)\geq 2$ or $|X^{\rm sing}|<4$.
\end{enumerate} 
\end{theorem} 

We do not know much more about the cluster type condition for singular surfaces. It would be a good step forward to find criteria for a singular surface to be of cluster type.

\begin{problem}\label{prob:ct-sing}
Let $X$ be a Gorenstein Fano surface with only $A_n$ singularities. Find a criterion for $X$ to be of cluster type.
\end{problem}

The discussion of higher-dimensional cluster type varieties 
will benefit from understanding some notions from singularity theory and birational geometry.
The following two sections will focus on explaining notions about singularities of the MMP and birational geometry of pairs, respectively.

\section{Singularity theory}
\label{sec:sing-theory} 

In this section, we recall some concepts about singularities of pairs. 
First, we focus on giving the important definitions, and then we will explore several examples of interest.

\begin{definition} 
{\em 
A {\em pair} $(X,\Delta)$ is a couple consisting of a normal quasi-projective variety $X$ and an effective divisor $\Delta$ such that $K_X+\Delta$ is $\qq$-Cartier, i.e., $m(K_X+\Delta)$ is a Cartier divisor for a suitable integer $m$. 
}
\end{definition} 

If $\Delta=0$, then a pair $(X,0)$ is nothing but a $\qq$-Gorenstein variety, and in this case we suppress the boundary from the notation. 
In algebraic geometry, we are often interested in pairs rather than varieties. 
Pairs appear naturally in many settings: via the Riemann-Hurwitz formula when taking finite quotients, via adjunction, and to deal with non-$\qq$-Gorenstein varieties. In the latter case, the boundary $\Delta$ can be regarded as a correction term for the $\qq$-Gorenstein condition.

\begin{definition}
{\em 
Let $(X,\Delta)$ be a pair. 
Let $\pi\colon Y\rightarrow X$ be a projective birational morphism from a normal variety $Y$.
Let $E\subset Y$ be a prime divisor. 
We write 
\[
\pi^*(K_X+\Delta)=K_Y+\Delta_Y,
\]
where $K_Y$ is a canonical divisor on $Y$ for which 
$\pi_*K_Y=K_X$. 
Here, $\pi^*(K_X+\Delta)$ is the pull-back of $\qq$-Cartier divisors.
The {\em log discrepancy} of $(X,\Delta)$ at $E$ is defined to be 
\[
a_E(X,\Delta)=1-{\rm coeff}_E(\Delta_Y).
\]
The log discrepancy is a rational number. 
Such a prime divisor $E\subset Y$ is called a {\em prime divisor} over $X$, as it naturally induces a divisorial valuation on $X$. 
}
\end{definition}

Philosophically, the log discrepancy of $(X,\Delta)$ with respect to $E$ measures how singular is the pair $(X,\Delta)$ along the tangent directions on $X$ corresponding to points in $E$.
The larger the log discrepancies, the better we expect the singularities of $(X,\Delta)$ to be. 
This motivates the following definition.

\begin{definition}
{\em 
We say that a pair $(X,\Delta)$ is {\em log terminal} (resp. {\em canonical} and {\em log canonical})
if $a_E(X,\Delta)>0$ (resp. $a_E(X,\Delta)\geq 1$ and $a_E(X,\Delta)\geq 0$) for every prime divisor $E$ over $X$. 
The previous definition of log terminal is also known as {\em Kawamata log terminal} or {\em klt} in the literature. 
We say that a pair $(X,\Delta)$ is {\em terminal}
if $a_E(X,\Delta)> 1$ for every prime divisor $E$ which is exceptional over $X$. 
Let $\epsilon>0$ be a positive real number.
We say that a pair $(X,\Delta)$ is {\em $\epsilon$-log canonical} if $a_E(x,\Delta)\geq \epsilon$ for every prime divisor $E$ over $X$. 
We may abbreviate log canonical and $\epsilon$-log canonical by lc and $\epsilon$-lc, respectively.
}
\end{definition} 

\begin{definition}\label{def:klt-type}
{\em 
A variety $X$ is said to be {\em klt type} (resp. {\em lc type}) if there exists a boundary $\Delta$ such that the pair $(X,\Delta)$ is klt (resp. lc).
}
\end{definition}

\begin{definition}
{\em  
Let $(X,\Delta)$ be a pair.
A {\em non-klt place} of $(X,\Delta)$ is a prime divisor $E$ over $X$ for which $a_E(X,\Delta)\leq 0$. The image of a non-klt place $E$ on $X$ is called a {\em non-klt center} of the pair.
An {\em lc place} of $(X,\Delta)$ is a prime divisor
$E$ over $X$ for which $a_E(X,\Delta)=0$. 
The image of an lc place $E$ on $X$ is said to be an {\em lc center} of the pair.
}
\end{definition} 

\begin{definition}
{\em 
A pair $(X,\Delta)$ is said to be {\em divisorially log terminal} or {\em dlt} if $(X,\Delta)$ is lc and the pair $(X,\Delta)$ is simple normal crossing\footnote{this means that $X$ is smooth and the components of $\Delta$ are transversal.} near every log canonical center of $(X,\Delta)$.

Let $(X,\Delta)$ be a log canonical pair.
A {\em dlt modification} of $(X,\Delta)$ is a projective birational morphism $p\colon Y\rightarrow X$ satisfying the following conditions:
\begin{enumerate}
\item $p$ only extract log canonical places of $(X,\Delta)$; and 
\item the pair $(Y,\Delta_Y)$, defined by 
$\pi^*(K_X+\Delta)=K_Y+\Delta_Y$, is dlt.
\end{enumerate} 
In the previous setting, we may say that 
$(Y,\Delta_Y)\rightarrow (X,\Delta)$ is a dlt modification.
}
\end{definition} 

It is known that every log canonical pair admits a dlt modification~\cite[Theorem 3.1]{KK10}.
For a log canonical pair, the log canonical centers are its most singular locus.
Thus, a dlt modification is a transformation that makes the most singular locus snc. 

\begin{definition}
{\em 
Let $(Y,\Delta_Y)$ be a dlt pair
and write $\Delta_Y^{=1}=\sum_{i\in I} B_i$. 
The {\em dual complex} of $(Y,\Delta_Y)$,
denoted by $\mathcal{D}(Y,\Delta_Y)$, is the regular $\triangle$-complex whose vertices are in correspondence with the prime components $B_i,i\in I$ and its $k$-cells are correspond to strata of $\sum_{i\in I}B_i$ of codimension $k+1$ in $Y$, i.e., prime components of $B_{i_0}\cap\dots\cap B_{i_k}$.
}
\end{definition} 

\begin{definition}
{\em 
Let $(X,\Delta)$ be a log canonical pair.
The {\em dual complex} of $(X,\Delta)$
is defined to be $\mathcal{D}(Y,\Delta_Y)$ 
where $(Y,\Delta_Y)\rightarrow (X,\Delta)$ is a dlt modification. 
}
\end{definition}

A priori, the dual complex of a log canonical pair depends on a given dlt modification.
However, by~\cite[Proposition 11]{dFKX}, the dual complex of a log canonical pair is well-defined up to simple homotopy equivalence.

\begin{definition}
{\em 
Let $(X,\Delta)$ be a log canonical pair.
The {\em regularity} of $(X,\Delta)$ denoted by 
${\rm reg}(X,\Delta)$ is the maximum of the dimension of $\mathcal{D}(Y,\Delta_Y)$ where
$(Y,\Delta_Y)\rightarrow (X,\Delta)$ is a dlt modification. 
The {\em coregularity} of a pair $(X,\Delta)$ is defined to be 
\[
{\rm coreg}(X,\Delta):=\dim X - {\rm reg}(X,\Delta)-1.
\]
For a log canonical pair, the coregularity is the dimension of the smallest log canonical center in a dlt modification.
The {\em absolute coregularity} of a klt pair $(X,\Delta)$ is
\[
\hat{\rm coreg}(X,\Delta):= \min 
\{ {\rm coreg}(X,B) \mid B\geq \Delta \text{ and } 
(X,B) \text{ is log canonical} \}. 
\]
}
\end{definition}

All the previous definitions are also valid in the local setting, i.e., near a closed singular point.
Whenever we work in the local setting, we will write $(X;x)$ to mean a variety $X$ near a closed point $x\in X$. 

The coregularity of a Calabi--Yau pair is closely related to its  fundamental groups. The following theorem is a consequence of the work of Braun and Figueroa (see~\cite[Theorem 2]{BF24}). 
It is a complementary result to Proposition~\ref{prop:fundamental-ct}.

\begin{theorem}\label{thm:Braun+Figueroa}
Let $X$ be a cluster type variety with klt type singularities. 
Then, the fundamental group $\pi_1(X^{\rm sm})$ is finite. 
\end{theorem}

\begin{example}
{\em 
In the case of dimension two, the previous classes of singularities have been completely classified (see, e.g.,~\cite[Section 3]{Kol92}). 
A surface $X$ is terminal if and only if it is smooth. 
A surface $X$ is canonical if and only if it has $A,D,E$-singularities, i.e., it is locally the quotient of $\cc^2$ by a finite subgroup $G\leqslant {\rm SL}(2,\cc)$. 
A surface $X$ is log terminal if and only if it has finite quotient singularities.
From the perspective of the absolute coregularity, for surfaces, there are only two possible values: zero and one.
Log terminal surface singularities of absolute coregularity zero are either toric or involution quotient of toric singularities~\cite[Theorem 10]{FFMP22}.
Log terminal singularities of absolute coregularity one are called {\em exceptional}, and the graphs of their minimal resolutions correspond to the Dynkin diagrams $E_6$, $E_7$, and $E_8$.
Finally, an elliptic singularity is an example of a log canonical surface singularity that is not log terminal.
}
\end{example}

\begin{example}
{\em 
In the case of dimension three, the class of terminal threefold singularities has been classified by the work of Reid and Mori (see, e.g.,~\cite{Rei87}). 
More precisely, given a  terminal threefold singularity 
$(X;x)$, we can find a finite cover $(Y;y)\rightarrow (X;x)$ only ramifying at $x\in X$ such that $(Y;y)$ is a {\em compound Du Val singularity}, i.e., a hypersurface singularity whose general hyperplane sections are Du Val.
The order of the finite index $Y\rightarrow X$ equals the {\em Cartier} index of $K_X$, i.e., the smallest $r$ such that $rK_X$ is Cartier near the singular point. 
It is expected that many canonical threefold singularities are complete intersections in toric ambient spaces (see~\cite[Conjecture]{Rei80}).
There is a handful of instances of canonical threefold singularities that have been classified: the toric case, the case of complexity one~\cite{BH20}, and some hypersurfaces~\cite{IP01}.
However, we still do not have a good understanding of canonical threefold singularities.
}
\end{example} 

In~\cite[Theorem 1]{Kol11}, Koll\'ar shows a sequence of terminal singularities, with no interesting quasi-\'etale finite covers, whose embedding dimension diverges. This is in contrast to terminal threefold singularities, which admit finite covers that are hypersurface singularities. This shows that the understanding of terminal fourfold singularities is already much more challenging than that of terminal threefolds.
Similar examples are given in~\cite[Theorem 2]{Kol11} for log canonical threefold singularities; the dual complex of such a surface can be homotopic to an arbitrary Riemann surface.

In the next two examples, we study two classes of singularities that are diametrically opposed:
exceptional singularities and toric singularities.

\begin{example}{\em 
A klt singularity of dimension $n$ is said to be {\em exceptional} if its absolute coregularity equals $\dim X-1$.
This is the maximum possible coregularity for a $n$-dimensional klt singularity.
In other words, a $n$-dimensional klt singularity $(X;x)$ is exceptional if for any boundary divisor $B$ through $x\in X$ for which $(X,B;x)$ is log canonical, the pair $(X,B)$ has a unique log canonical place. Among Du Val singularities these are precisely the $E_6,E_7$, and $E_8$ singularities. Higher-dimensional exceptional singularities are still hard to classify, however, we have a good qualitative understanding of these singularities. 
A consequence of~\cite[Theorem 1.1]{HLM20} is that for $\epsilon>0$ the class of $n$-dimensional $\epsilon$-lc exceptional singularities is bounded up to degenerations. 
This means that there is a finite type family $\mathcal{S}_{n,\epsilon}$ of singularities such that every
$n$-dimensional
$\epsilon$-lc exceptional singularity degenerates to a singularity in $\mathcal{S}_{n,\epsilon}$.
In particular, most of the invariants of singularities can be controlled for $n$-dimensional $\epsilon$-lc exceptional singularities. 
For instance, given any such a singularity $(X;x)$,
we have $|\pi_1(X;x)|\leq f(n,\epsilon)$ where 
$\pi_1(X;x)$ denotes the fundamental group of the singularity 
and $f(n,\epsilon)$ is a constant that only depends on the dimension and the mld. 
More generally, for a $n$-dimensional exceptional singularity $(X;x)$, we have a short exact sequence 
\begin{equation}\label{eq:ses-exceptional-fun}
1\rightarrow \zz/m\zz \rightarrow \pi_1(X;x) \rightarrow G \rightarrow 1, 
\end{equation} 
where $G$ is a finite group of order at most $f(n)$ (see, e.g.,~\cite[Theorem 5]{Mor21}). 
For a $n$-dimensional klt singularity $(X;x)$, we may find a projective birational morphism $\pi\colon Y\rightarrow X$, which is an isomorphism over $X\setminus \{x\}$, and $\pi^{-1}(x)=E$ is a Fano type variety (see Definition~\ref{def:ft}). 
If $(X;x)$ is an exceptional singularity, 
then the pair $(E,\Delta)$ obtained by adjunction of $(Y,E)$ to $E$ is an exceptional Fano pair. 
An exceptional Fano pair is a Fano pair whose absolute coregularity equals its dimension.
By the work of Birkar, these Fano pairs form a bounded family in each dimension (see~\cite[Theorem 1.3]{Bir19}). In the exact sequence~\eqref{eq:ses-exceptional-fun}, the subgroup $\zz/m\zz$ corresponds to a loop around $E$ while the quotient $G$ corresponds to the orbifold fundamental group of $(E,\Delta_E)$.

In summary, the combinatorial side of exceptional singularities is simple; we can find a partial resolution that extracts an exceptional Fano pair $(E,\Delta_E)$. However, the geometry of $(E,\Delta_E)$ itself can be complicated.  
}
\end{example} 

\begin{example}\label{ex:toric-mld}
{\em 
A toric singularity $(X;x)$ is the germ, near a closed torus invariant point, of a toric variety. 
If $(X;x)$ is a toric singularity and $B$ its reduced torus invariant boundary, then the pair $(X,B)$ is log canonical (see~\cite[Proposition 11.4.24]{CLS11}). 
Log canonical places of $(X,B)$ are in correspondence with torus invariant divisorial valuations of over $(X;x)$.
This pair structure shows that a toric singularity has absolute coregularity zero.
A dlt modification $(Y,B_Y)\rightarrow (X,B)$ is nothing but a resolution of singularities of the pair.
The dual complex $\mathcal{D}(X,B)$ of a $n$-dimensional toric singularity with its toric boundary is a simplicial structure on a $(n-1)$-dimensional disk.
The fundamental group of a $n$-dimensional toric singularity is a finite abelian group of rank at most $n-1$ (see, e.g.,~\cite[Proposition 12.1.9]{CLS11}). 

We explain how to compute the log discrepancies of a $\qq$-Gorenstein toric singularity. 
A $n$-dimensional toric singularity $(X;x)$ corresponds to a pointed polyhedral cone $\sigma \subset \qq^n$. 
Any integral point of $\sigma$ corresponds to a toric valuation over $(X;x)$. The points in ${\rm relint}(\sigma)\cap \zz^n$ correspond to exceptional torus invariant prime divisors with center in $x\in X$.
Since we are assuming that $(X;x)$ is $\qq$-Gorenstein, there is a linear function $L$ that takes the value one on each primitive generator of extremal rays of $\sigma$ (see~\cite[Section 1]{Amb06}).
The log discrepancy of $(X;x)$ at the toric valuation corresponding to $v\in {\rm relint}(\sigma)\cap \zz^n$ is nothing but the value of $L$ at $v$.
One can often find a resolution of singularities of $(X;x)$ by only extracting divisors which correspond to generators of the monoid $\sigma \cap \zz^n$. 

Toric surface singularities can be classified via their log discrepancies (see, e.g.~\cite{Amb24}).
We review two classic examples of toric surface singularities and study their log discrepancy functions.
The $A_n$ singularity corresponds to the cone 
$\sigma_n:=\langle (0,1),(n+1,1)\rangle$ and the divisors extracted in its minimal resolution correspond to 
$(k,1)$ with $k\in \{1,\dots, n\}$. The log discrepancy function takes value one on all the lattice points of the form $(k,1)$. On the other hand, the cone over a rational curve of degree $n$, let's call it $C_n$, corresponds to the cone 
$\tau_n:=\langle (0,1),(2n,2n-1)\rangle $. It can be resolved by a single blow-up extracting the divisor corresponding to the lattice point $(1,1)$ at which the log discrepancy function has value $\frac{1}{2n}$. The cone over a rational curve of degree $n$ is precisely the only singularity in the weighted projective space $\pp(1,1,n)$. Therefore, we conclude that ${\rm mld}(\pp(1,1,n))=\frac{1}{2n}$.

In summary, the combinatorial side of toric singularities is rich, however, the geometry of the exceptional divisors needed to resolve such singularities tends to be simple,
similar to such of projective spaces.
}
\end{example}

Let $X$ be a klt type variety. 
Given a cluster type pair $(X,B)$ the singularities of $X$ along $X\setminus B$ are Gorenstein canonical singularities. 
However, the actual class of singularities seems to be much smaller. For instance, in the case of surfaces with orbifold singularities $X\setminus B$ has only $A_n$ singularities, rather than $D_n$ or $E_6,E_7$, and $E_8$ singularities. 
The singularities of cluster varieties is classified in the works of Benito, Faber, Mourtada, and Schober~\cite{BFMS23,BFMS24}. In such case, the singularities can be described rather explicitly and they are mostly compound Du Val singularities. The following is a natural problem.

\begin{problem}\label{prob:sing-dim-3-ct}
Classify the singularities of $X\setminus B$ where $(X,B)$ is a three-dimensional cluster type pair. 
\end{problem} 

\section{Birational geometry} 
\label{sec: bir-geom}

In this section, we discuss some basic notions regarding birational geometry. We mostly discuss birational geometry of Fano varieties and Calabi--Yau pairs. We start introducing some notions related to the global curvature of projective varieties.

\begin{definition}\label{def:ft}
{\em 
A normal projective variety $X$ is said to be {\em Fano} if $-mK_X$ is a very ample divisor for some integer $m$ and $X$ has klt singularities.\footnote{Some authors call these varieties {\em klt Fano} varieties. However, all the Fano varieties we consider will be klt unless otherwise stated.}
A variety $X$ is said to be {\em Fano type} if there exists a boundary $\Delta$ in $X$ for which 
$(X,\Delta)$ has klt singularities and $-(K_X+\Delta)$ is big and nef.
}
\end{definition}

By~\cite[Corollary 1.3.1]{BCHM10}, we know that a Fano type variety is a Mori dream space. In particular, the MMP for any divisor on a Fano type variety terminates with a good minimal model or a Mori dream space. Fano varieties play a central role in the classification of algebraic varieties and, as observed above, also in the understanding of log terminal singularities.

\begin{definition}\label{def:cy-type}
{\em 
A pair $(X,\Delta)$ is said to be a {\em Calabi--Yau} 
if $K_X+\Delta\sim_\qq 0$ and $(X,\Delta)$ has log canonical singularities.
A pair $(X,\Delta)$ is said to be {\em Calabi--Yau type} if there exists a boundary $B\geq \Delta$ for which 
$(X,B)$ is Calabi--Yau.
}
\end{definition} 

Every Fano variety is Calabi--Yau type. Indeed, if $X$ is Fano and we choose $m$ such that $|-mK_X|$ is very ample, then for $\Gamma \in |-mK_X|$ general enough the pair
$(X,\Gamma/m)$ has log canonical singularities. On the other hand, it is often the case that Calabi--Yau varieties can be degenerated into reducible varieties whose components are Fano (or Fano type). An example of this phenomenon are elliptic curves in $\pp^2$ degenerating into the toric boundary which is a cycle of three rational curves. 

\begin{definition}
{\em 
Let $(X,\Delta)$ and $(Y,\Delta_Y)$ be two pairs. 
We say that these pairs are {\em crepant birational}
or {\em crepant birational equivalent} if the following conditions hold:
\begin{enumerate}
\item there is a birational map $\pi \colon Y \dashrightarrow X$;
\item a resolution of indeterminancy $p\colon Z \rightarrow Y$ and $q\colon Z\rightarrow X$ of $\pi$; and 
\item we have $q^*(K_X+\Delta)=p^*(K_Y+\Delta_Y)$. 
\end{enumerate}
In the previous setting, we may write 
$(X,\Delta) \simeq_{\rm cbir} (Y,\Delta_Y)$.
}
\end{definition} 

By definition, for every divisor $E$ over $X$, we have $a_E(X,\Delta)=a_E(Y,\Delta_Y)$. 
Thus, crepant birational pairs have the same log discrepancies. 
The word {\em crepant} does not really exist in any language. However, the prefix {\em dis} expresses negation, and crepant pairs have the same discrepancies. 
Crepant birational pairs also share some other properties as we see in the following lemma. 

\begin{proposition}\label{prop:crep-bir-prop}
Let $(X,\Delta)$ and $(Y,\Delta_Y)$ be two crepant birational pairs. Then, the following statements hold:
\begin{enumerate}
\item $(X,\Delta)$ and $(Y,\Delta_Y)$ have the same Cartier index;
\item $(X,\Delta)$ is Calabi--Yau if and only if 
$(Y,\Delta_Y)$ is Calabi--Yau;
\item $(X,\Delta)$ and $(Y,\Delta_Y)$ have the same coregularity; and 
\item $(X,\Delta)$ and $(Y,\Delta_Y)$ have the same dual complex up to simple homotopy equivalence. 
\end{enumerate} 
\end{proposition} 

The proof of Proposition~\ref{prop:crep-bir-prop}.(1) and (2) follow from~\cite[Lemma 3.1]{FMM22}. The proof of Proposition~\ref{prop:crep-bir-prop}.(4) is a consequence of~\cite[Proposition 11]{dFKX}. The third statement in Proposition~\ref{prop:crep-bir-prop} follows from the fourth statement.

\begin{definition}\label{def:log-rational}
{\em 
Let $(X,B)$ be a pair.
We say that $(X,B)$ is {\em log rational} if there exists a crepant birational map 
$(\pp^n,H_0+\dots+H_n)\dashrightarrow (X,B)$ where 
the $H_i$'s are the coordinate hyperplanes of the projective space. We say that $X$ is a {\em log rational} variety if there exists a boundary $B$ for which $(X,B)$ is a log rational pair. 
}
\end{definition} 

By Proposition~\ref{def:log-rational}, we know 
that a log rational pair of dimension $n$ is Calabi--Yau, has index one, coregularity zero, and $\mathcal{D}(X,B)\simeq_{\rm PL} S^{n-1}$.
If $T_1$ and $T_2$ are two complete $n$-dimensional toric varieties and $B_1$ and $B_2$ are its reduced toric invariant boundaries, then $(T_1,B_1)\simeq_{\rm cbir}(T_2,B_2)$ holds.
A projective birational torus equivariant morphism
$T_1\rightarrow T_2$ between the toric varieties $T_1$ and $T_2$, corresponds to a refinement of the fan $\Sigma_2$ of $T_2$ into the fan of $\Sigma_1$ (see~\cite[Example 3.4.9]{CLS11}).
The blow-up of a toric strata corresponds to the star subdivision of the corresponding cone~\cite[Proposition 3.3.15]{CLS11}. 
The crepant birational equivalence $(T_1,B_1)\simeq_{\rm cbir} (T_2,B_2)$
follows from the fact that two fans with the same support
always admit a common refinement.
Recently, Adiprasito and Pak proved the strong factorization conjecture for toric varieties which gives a much stronger statement~\cite{AP24}. Any two triangulations admit a common stellar subdivision.
This means that for any two smooth projective toric varieties of the same dimension $T_1$ and $T_2$
there are projective morphisms $p\colon T\rightarrow T_1$
and $q\colon T\rightarrow T_2$ where both $p$ and $q$ are compositions of blow-ups of torus invariant subvarieties. Thus, a pair $(X,B)$ is log rational
if and only if it is crepant birational to a toric pair.

In the case of dimension two, we can check log rationality using following proposition. 

\begin{proposition}\label{prop:log-rat-dim-2}
Let $(X,B)$ be a log pair of dimension two. 
The pair $(X,B)$ is log rational if and only if 
$K_X+B\sim 0$ and $\mathcal{D}(X,B)\simeq_{\rm PL} S^1$.
\end{proposition}

\begin{proof}
One direction is clear from Proposition~\ref{prop:crep-bir-prop}. Now, assume that $K_X+B\sim 0$ and $\mathcal{D}(X,B)\simeq_{\rm PL} S^1$.
Let $p\colon (Y,B_Y)\rightarrow (X,B)$ be a dlt modification, then $Y$ has finite quotient singularities and it is smooth along $B_Y$.
As $K_Y+B_Y\sim 0$, we conclude that the singularities of $Y$ are $A,D,E$-singularities. 
Thus, $Y$ admits a crepant resolution $q\colon Z\rightarrow Y$.
Let $q^*(K_Y+B_Y)=K_Z+B_Z$. 
By construction, we know that $Z$ is smooth, $K_Z+B_Z\sim 0$, and $\mathcal{D}(Z,B_Z)\simeq_{\rm PL} S^1$. 
Therefore, by Proposition~\ref{prop:Looijenga-pairs-are-cluster-type}, we conclude that $(Z,B_Z)$ is cluster type and so it is log rational. We conclude that $(X,B)$ is log rational as $(X,B)$ and $(Z,B_Z)$ are crepant birational equivalent.
\end{proof}

In other words, a pair $(X,B)$ is log rational if and only if $B$ is an anti-canonical nodal curve. Thus, from Proposition~\ref{prop:Looijenga-pairs-are-cluster-type} and Proposition~\ref{prop:log-rat-dim-2}, we conclude that for smooth surfaces the concepts of cluster type pairs and log rational pairs are the same. In general, there are log rational pairs which are not cluster type (see Example~\ref{ex:ct-sing-surf}). The following is the definition of cluster type pairs from the birational perspective. 

\begin{definition}\label{def:ct}
{\em 
Let $(X,B)$ be a pair.
We say that $(X,B)$ is {\em cluster type} if there exists a crepant birational map 
\begin{equation}\label{eq:crepant-bir} 
\phi\colon (\pp^n,H_0+\dots+H_n)\dashrightarrow (X,B)
\end{equation} 
such that ${\rm Ex}(\phi)\cap \mathbb{G}_m^n$ has codimension at least two in $\mathbb{G}_m^n$. Here, the $H_i$'s are the coordinate hyperplanes of the projective space $\mathbb{P}^n$.
We say that a variety $X$ is of {\em cluster type}
if there exists a boundary $B$ for which $(X,B)$ is cluster type. In this case, we say that $(X,B)$ is a {\em cluster type pair} and that $B$ is a {\em cluster type boundary}.
}
\end{definition} 

In other words, the birational map~\eqref{eq:crepant-bir} can only contract the divisors $H_0,\dots,H_n$.
The following definition is a generalization of the concept of strong toric models given above. 

\begin{definition}\label{def:toric-model}
{\em A {\em toric model} of a Calabi--Yau pair $(X,B)$, is a crepant birational map $(T,B_T)\dashrightarrow (X,B)$ where $(T,B_T)$ is a toric Calabi--Yau pair.}
\end{definition}

We conclude this section by arguing that a log rational pair is very close to being of cluster type.

\begin{proposition}\label{prop:from-lr-to-ct}
Let $(X,B)$ be a log rational pair. 
Then, there exists a crepant birational morphism
$(Y,B_Y)\rightarrow (X,B)$, where $(Y,B_Y)$ is cluster type. 
\end{proposition} 

\begin{proof}
Let $\phi\colon (\pp^n,H_0+\dots+H_n)\dashrightarrow (X,B)$ be a crepant birational map. 
Let $E_1,\dots,E_k$ be the closure in $\mathbb{P}^n$
of the divisors in ${\rm Ex}(\phi)\cap \mathbb{G}_m^n$. 
As these pairs are crepant, we have 
\[
1=a_{E_i}(\pp^n,H_0+\dots+H_n) = 
a_{E_i}(X,B),
\]
for every $i\in \{1,\dots,k\}$. 
Let $(X_1,B_1)\rightarrow (X,B)$ be a $\qq$-factorial dlt modification of $(X,B)$. 
Note that 
\[
a_{E_i}(X,B)=a_{E_i}(X_1,B_1)
\]
for every $i\in \{1,\dots,k\}$. 
Let $Z_1,\dots,Z_k$ be the centers on $X_1$ of the prime divisors $E_1,\dots,E_k$. 
Passing to a higher dlt modification, we may assume that 
no $Z_i$ equals a strata of $B_1$.
Let $\Gamma_1$ be an effective divisor on $X_1$ that contains $Z_1\cup \dots\cup Z_k$. 
For $\epsilon>0$ small enough, the pair $(X_1,B_1+\epsilon \Gamma_1)$ is dlt and 
\[
a_{E_i}(X_1,B_1+\epsilon \Gamma_1) <1 
\]
for every $i\in \{1,\dots,k\}$. 
By~\cite[Theorem 1]{Mor19}, there is a projective birational morphism $Y\rightarrow X_1$ that extracts precisely the divisors $E_1,\dots,E_k$. 
Write $p\colon Y\rightarrow X$ for the composition.
Therefore, we have $p^*(K_X+B)=K_Y+B_Y$ where $(Y,B_Y)$ is a Calabi--Yau pair. By construction, we have a crepant birational map 
\[
\psi\colon (\pp^n,H_0+\dots+H_n)\dashrightarrow (Y,B_Y)
\]
that contracts no divisors in $\mathbb{G}_m^n$.
Therefore, the pair $(Y,B_Y)$ is cluster type.
Thus, we have a crepant birational map 
$(Y,B_Y)\rightarrow (X,B)$ from a cluster type pair.
\end{proof} 

In the case of surfaces, given a log rational surface $(X,B)$, we can obtain a cluster type surface by resolving the singularities of $X\setminus B$. This statement follows from Proposition~\ref{prop:Looijenga-pairs-are-cluster-type}.
In the following section, we discuss log rational and cluster type pairs of dimension three.

\section{The cluster type condition in dimension three}
\label{sec:dim-3}
In this section, we discuss log rational and cluster type pairs of dimension three. 
The first question is whether a Calabi--Yau pair $(X,B)$ of dimension three is log rational or not.
By Proposition~\ref{prop:crep-bir-prop}, we know that such 
threefold pair must satisfy $K_X+B\sim 0$ and $\mathcal{D}(X,B)\simeq_{\rm PL} S^2$. 
In the case of surfaces, controlling the index and the dual complex suffices to check whether a pair is log rational (Proposition~\ref{prop:log-rat-dim-2}).
However, in dimension three and higher, this is not enough.
In~\cite{Ka20}, Kaloghiros gives an example of a Calabi--Yau threefold pair $(X,B)$ satisfying the following three conditions:
\begin{itemize}
\item $K_X+B\sim 0$;
\item $\mathcal{D}(X,B)\simeq_{\rm PL} S^2$; and 
\item $X$ is a birationally rigid variety.
\end{itemize} 
Therefore, the pair $(X,B)$ is not log rational.
Indeed, if $(X,B)$ is log rational, then $X$ must be a rational variety. In~\cite[Theorem 1.2.1]{Duc24}, Ducat proved the following theorem. 

\begin{theorem}\label{thm:ducat-dim-3}
Let $(\pp^3,B)$ be a Calabi--Yau pair of index one with 
$\mathcal{D}(\pp^3,B)\simeq_{\rm PL} S^2$. Then, the pair $(\pp^3,B)$ is log rational.
\end{theorem} 

In the case of dimension two 
we have three possible Calabi-Yau pairs $(\pp^2,B)$ 
of index one and coregularity zero. 
These three pairs can be related by Cremona transformations.
Indeed, if $(\pp^2,C_3)$ is a pair consisting of the projective plane and a nodal cubic, then we can perform 
a Cremona transformation $\phi_1\colon \pp^2\dashrightarrow \pp^2$ based at the node of $C_3$ and two other points in $C_3$. By doing so, we obtain a crepant birational map 
\[
\phi_1\colon (\pp^2,C_3) \dashrightarrow (\pp^2,Q_2+L)
\]
where $Q_2$ is a conic and $L$ is a transversal line. 
If we perform a second Cremona transformation 
$\phi_2\colon \pp^2 \dashrightarrow \pp^2$ based on one of the intersection points $Q_2\cap L$ and two other points of $Q_2\setminus L$, then we obtain a crepant birational map
\[
\phi_2\colon (\pp^2,Q_2+L)\dashrightarrow (\pp^2,H_0+H_1+H_2),
\]
where the $H_i$'s are the coordinate hyperplanes of $\pp^2$. 
This argument shows that any nodal curve $B\in |-K_{\pp^2}|$ can be transformed into the toric boundary using Cremona transformations. 
In~\cite{Duc24}, Ducat analyzes the quartic surfaces
$B\in |-K_{\pp^3}|$ of coregularity zero 
and uses Cremona transformations and cubo-cubic transformations 
to transform this quartic into the toric boundary of $\pp^3$.
Ducat proposed the following conjecture.

\begin{conjecture}\label{conj:log-rat-dim-3}
Let $(X,B)$ be a pair of dimension three. Assume that 
\begin{enumerate}
\item $X$ is a rational variety; 
\item $K_X+B\sim 0$; and 
\item $\mathcal{D}(X,B)\simeq_{\rm PL} S^2$. 
\end{enumerate} 
Then, the pair $(X,B)$ is log rational.
\end{conjecture} 

The previous provides a conjecturally satisfactory answer to the problem of log rationality in dimension three.
This conjecture has been checked in a few cases.
In~\cite{LMV24}, Loginov, Vasilkov, and the author checked the conjecture for some boundaries on smooth rational Fano threefolds. 
More precisely, they proved the following theorem.

\begin{theorem}\label{thm:log-rat-smooth-Fano}
Let $X$ be a general rational smooth Fano threefold. 
Then, the variety $X$ is log rational. 
\end{theorem} 

In the previous statement, we say that a variety is {\em general} if it is general in its deformation space. 
By the work of Iskovskikh, Prokhorov, Mori, and Mukai, there are $105$ families of smooth Fano threefolds~\cite{IP99,MM81}. 
From these, there are $88$ families whose general element is a rational variety. 
Theorem~\ref{thm:log-rat-smooth-Fano} states that a general element of each such $88$ families is a log rational variety.
Note that the {\em general} word in the previous statement is already needed in dimension two.
Indeed, some special del Pezzo surfaces of degree one are not log rational. 
It is still unclear which smooth Fano threefolds are cluster type. 

Conjecture~\ref{conj:log-rat-dim-3} is known for $X\simeq \pp^3$ by the work of Ducat~\cite{Duc24} and in some cases of pairs $(X,B)$ where $X$ is a smooth Fano threefold by~\cite{LMV24}. We propose the following problem.

\begin{problem}\label{prob:toric-case-log-rationality}
Let $(T,B_T)$ be a Calabi--Yau pair of index one and coregularity zero where $T$ is a toric threefold. Prove that the pair $(T,B_T)$ is log rational.
\end{problem}

The previous problem can be regarded as a toric case of the three-dimensional log rationality conjecture~\ref{conj:log-rat-dim-3}. This seems to be the most natural next case to tackle.

In some sense, the relation of cluster type pairs and log rational pairs is similar to that of uniformly rational and rational varieties. A $n$-dimensional variety $X$ is said to be {\em uniformly rational} if for every point $x\in X$ there is a neighborhood $x\in U\subseteq X$ which is isomorphic to an open subset $V\subseteq \pp^n$. It is an open question to decide whether every rational variety is indeed rational variety. In dimension two, every rational variety is uniformly rational. By the work of Cutkosky~\cite{Cut07}, and Adiprasito-Pak~\cite{AP24}, the following statement holds.

\begin{theorem}\label{thm:uniformization}
Let $X$ be a smooth rational threefold.
Then, there exists a projective birational morphism
$Y\rightarrow X$ which is a composition of blow-ups of smooth subvarieties, such that $Y$ is uniformly rational.
\end{theorem} 

In any dimension, a log rational pair admits a volume-preserving blow-up which is cluster type. 
In dimension three, every rational threefold admits a blow-up which is a uniformly rational threefold. 
This draws a similarity between the concepts of cluster type varieties and uniformly rational varieties.

In~\cite{AdSFM24}, Alves da Silva, Figueroa, and the author study whether a pair $(\pp^3,B)$ is of cluster type. 
A partial answer is given by the following theorem.

\begin{theorem}\label{ct-in-p3}
Let $(\pp^3,B)$ be a Calabi--Yau pair of index one and coregularity zero. Assume that $B$ is non-normal. 
If the non-normal locus of $B$ is not contained in a plane, then $(\pp^3,B)$ is of cluster type.
\end{theorem} 

The proof of Theorem~\ref{ct-in-p3} mixes modern tools from birational geometry of cluster type pairs with classic tools about birational geometry of quartic threefolds. 
For instance, many results on non-normal quartic surfaces due to Jessop are used (see~\cite{Jes16}).
When $B$ is reducible, the authors give a complete characterization of cluster type pairs.

\begin{theorem}\label{ct-in-p3-red}
Let $(\pp^3,B)$ be a Calabi--Yau pair of index one and coregularity zero. Assume that $B$ is reducible.
Then, the pair $(\pp^3,B)$ is of cluster type unless $B$ is the sum of a hyperplane and a cubic surface intersecting smoothly along a nodal plane cubic.
\end{theorem} 

It is still unclear how to classify cluster type pairs of the form $(\pp^3,B)$. Or equivalently, how to classify quartic surfaces which admit an algebraic torus $\mathbb{G}_m^3$ in their complements. The following problem is still open.

\begin{problem}\label{prob:quartic-surfaces}
Classify quartic surfaces $B\subset \pp^3$ for which there is an embedding $\mathbb{G}_m^3 \hookrightarrow \pp^3\setminus B$.
\end{problem} 

\section{Complexity, birational complexity, and the cluster type condition} 
\label{sec:complexity}

In this section, we review the concepts of complexity and birational complexity. Then we connect the complexity to toric geometry and cluster geometry.

\begin{definition}\label{def:comp} 
{\em 
Let $(X,\Delta)$ be a pair.
Write $\Delta=\sum_{i=1}^n a_iP_i$, where the $P_i$'s are the prime components of $\Delta$ and each $a_i\geq 0$.
The {\em complexity} of $(X,\Delta)$ is defined to be 
\[
c(X,\Delta):= \dim X + \dim {\rm Cl}_\qq(X) - |\Delta|,
\]
where $|\Delta|=\sum_{i=1}^n a_i$.
The {\em birational complexity} of $(X,\Delta)$ is defined to be 
\[
c_{\rm bir}(X,\Delta):=\inf \{  c(X',\Delta') \mid 
(X',\Delta')\simeq_{\rm cbir} (X,\Delta) \}.
\]
If $r$ is the Cartier index of $K_X+\Delta$, then 
$rc_{\rm bir}(X,\Delta)\in \zz$. 
In particular, for a Calabi--Yau pair of index one the birational complexity is an integer.
}
\end{definition}

A similar definition can be given for singularities of pairs.

\begin{definition}
{\em 
Let $(X,\Delta;x)$ be a singularity of pairs, i.e., $x\in X$ is a closed point and $\Delta$ is an effective divisor passing through $x\in X$. We define the {\em local complexity} of $(X,\Delta)$ at $x$ as 
\[
c(X,\Delta;x):= \dim X + \dim {\rm Cl}_\qq(X_x) - |\Delta|_x, 
\]
where $X_x$ is the spectrum of the localization of $X$ at $x$ and $|\Delta|_x$ is the sum of the coefficients of the components of $\Delta$ passing through $x\in X$.
}
\end{definition} 

\begin{example}
{\em 
Let $(T,B_T)$ be a toric pair where $B_T$ is the reduced torus invariant divisor. 
Let $\Sigma$ be the fan associated to $T$.
Then, we have $|B_T|=|\Sigma(1)|$. 
On the other hand, we have $\dim {\rm Cl}_\qq(X)=|\Sigma(1)|-n$ (see, e.g.,~\cite[Theorem 4.1.3]{CLS11}).
We conclude that 
\[
c(T,B_T) = \dim T + \dim {\rm Cl}_\qq(T) - |B_T|=0.
\]
A similar computation shows that $c(T,B_T;t)=0$ whenever $t\in T$ is a torus invariant point of a toric variety
and $B_T$ is the reduced sum of all the torus invariant prime divisors.
}
\end{example} 

Thus, toric varieties and singularities have complexity zero when we equip them with their torus invariant divisor. 
The following theorem, proved by Brown, M\textsuperscript{c}Kernan, Svaldi, and Zong, gives the converse statement for complete varieties (see, e.g.,~\cite[Theorem 1.2]{BMSZ18}).

\begin{theorem}\label{thm:charact-toric}
Let $(X,B)$ be a Calabi--Yau pair. 
Then, we have $c(X,B)\geq 0$. 
Furthermore, if $c(X,B)<1$, then the variety $X$ is toric
and the divisor $\lfloor B\rfloor$ is torus invariant. 
\end{theorem} 

In~\cite{MS21}, Svaldi and the author prove a similar statement for singularities, i.e., a singularity of complexity zero is formally a toric singularity.
This gives a characterization of toric singularities
among log canonical singularities using the language of complexity.
In~\cite{MS21}, the authors also introduce the concept of relative complexity, for morphisms with connected fibers, and show that the lower bound zero is only attained by toric morphisms.

Note that we cannot conclude that $B$ is torus invariant in the setting of Theorem~\ref{thm:charact-toric} when $c(X,B)<1$, even if we assume $c(X,B)=0$. This is shown in the following example.

\begin{example}\label{ex:p1-non-toric}
{\em 
Let $p_1,\dots,p_{2k} \in \pp^1$ be different points.
Consider the pair $(\pp^1,\sum_{i=1}^{2k}\frac{1}{k}p_i)$. 
Then, we have $c(\pp^1,\sum_{i=1}^{2k}\frac{1}{k}p_i)=0$,
however, the sum of $2k$ points with coefficients $\frac{1}{k}$ is not a torus invariant divisor proivided $k\geq 2$. Indeed, for $k\geq 2$, the automorphism group of the pair
$(\pp^1,\sum_{i=1}^{2k}\frac{1}{k}p_i)$ is finite. 
Nevertheless, we can rewrite our pair as 
\[
\left( \pp^1, \sum_{i=1}^{k} \frac{1}{k} \left( 
p_{2i-1} + p_{2i}
\right)\right) 
\]
By doing so, we can think of our boundary as a weighted combination of toric boundaries, where the sum of the weights equals one. 
}
\end{example} 

The following theorem, due to Enwright and Figueroa, shows that the phenomenon from the previous example holds in general for Calabi--Yau pairs of complexity zero, independent of the index of such a pair. 

\begin{theorem}\label{thm:comp-zero}
Let $(X,B)$ be a Calabi--Yau pair of dimension $n$ and complexity zero. 
Then, there exists: 
\begin{enumerate} 
\item boundary divisors $B_1,\dots,B_k$ 
with $(X,B_i)$ a toric Calabi--Yau pair for each $i\in \{1,\dots,k\}$; and 
\item a partition of unity $1=\sum_{i=1}^k \mu_i$,
\end{enumerate}
such that $B=\sum_{i=1}^k \mu_i B_i$.
\end{theorem} 

The previous theorem gives a complete understanding of Calabi--Yau pairs of complexity zero. Some similar behaviour is expected for pairs with complexity less than one. 
In the case of complexity one, in~\cite[Section 7]{BMSZ18}, the authors give an example which is not rational; however, it becomes rational after a $2$-to-$1$ cover. 
An overarching generalization of this example was proved by Enwright, Li, and Y\'a\~nez in~\cite{ELY25}. One of the main results of their work is the following.

\begin{theorem}\label{thm:ely25}
Let $(X,\Delta)$ be a Calabi--Yau pair of complexity one. 
Then, there exists a $2$-to-$1$ cover $Y\rightarrow X$, only ramified along $X^{\rm sing}$ and an element $B \in |-K_X|$ such that the pair $(X,B)$ is cluster type. 
\end{theorem} 

Theorem~\ref{thm:ely25} shows that cluster type pairs appear naturally from the perspective of the complexity.
In this hierarchy, cluster type pairs are close to toric pairs. Further, in~\cite{ELY25}, the authors prove that for such $X$ as in Theorem~\ref{thm:ely25}, there exists an embedding $X\hookrightarrow T$ into a toric variety with $\dim T =\dim X+1$, and $B$ is the restriction of a torus invariant divisor on $T$. More generally, it is known that whenever we have a Calabi--Yau pair $(X,B)$ with $c(X,B)<2$, then $X$ is a toric hypersurface. In the case of complexity one, this toric hypersurface is a quadric in Cox coordinates. 
This leads naturally to the following problem.

\begin{problem}\label{prob:possible-complexities}
Describe the possible complexities of Calabi--Yau pairs in the interval $[0,2]$. Describe how these varieties embed into toric varieties. 
\end{problem} 

Some examples of local complexities of surfaces can be found in~\cite[Example 4.12]{Mor22}.
We expect only finitely many values in the interval $[0,2]$ for Problem~\ref{prob:possible-complexities}.
In view of~\cite{MS21}, it is expected that we can fully understand singularities of complexity less than two. 

\begin{problem}\label{prob:sing-comp}
Let $(X,B;x)$ be a log canonical singularity with $c(X,B;x)<2$. Describe the local geometry of $X$ and $B$ near $x$.
\end{problem} 

As a consequence of Theorem~\ref{thm:ely25} and Theorem~\ref{thm:charact-toric}, we conclude the following statement about the birational complexity.

\begin{theorem}\label{thm:bir-comp}
Let $(X,B)$ be a Calabi--Yau pair of index one. 
If $c_{\rm bir}(X,B)\leq 1$, then $c_{\rm bir}(X,B)=0$ and $(X,B)$ is log rational.
\end{theorem} 

The example due to Kaloghiros in~\cite{Ka20} shows that there are threefold Calabi--Yau pairs $(X,B)$ with $c_{\rm bir}(X,B)=2$. It is not clear which values of the birational complexity can be achieved. In~\cite{MM24}, Mauri and the author show that $c_{\rm bir}(X,B)<2\dim X$ whenever $(X,B)$ is Calabi--Yau and $X$ is a Fano type variety. 
More generally, the following theorem is proved.

\begin{theorem}\label{thm:bircomp-vs-coreg}
Let $X$ be a Fano type variety and $(X,B)$ be a Calabi--Yau pair. Then, we have $c_{\rm bir}(X,B)\leq \dim X + {\rm coreg}(X,B)$.
\end{theorem} 

In particular, in the context of Theorem~\ref{thm:bircomp-vs-coreg}, if the coregularity of the pair is zero, then the birational complexity is at most the dimension of $X$. 
The main result of~\cite{MM24} is about dual complexes of Calabi--Yau pairs. 

\begin{theorem}\label{thm:bircomp-vs-comp}
Let $(X,B)$ be a Calabi--Yau pair with $c_{\rm bir}(X,B)<\dim X$. If $\mathcal{D}(X,B)$ is a PL-manifold, then it is either a PL-disk or a PL-sphere. 
\end{theorem} 

In the previous theorem, we say that a dual complex is a PL-manifold if every link of the dual complex is a PL-sphere. 
Of course, the case in which it is harder to understand the dual complex is when $c_{\rm bir}(X,B)=\dim X$. 
If $X$ is a Fano variety of Picard rank one, then this happens precisely when $B$ has a single irreducible component. 

\section{Finite actions on Fano varieties}
\label{sec:finite-actions}

In this section, we recollect some results regarding finite actions on Fano varieties.
First, we recall the following theorem, due to Prokhorov and Shramov, about finite actions on Fano varieties (see, e.g.,~\cite{PS16}). 

\begin{theorem}\label{thm:jordan-property}
There exists a constant $i_0:=i_0(n)$, only depending on $n$, satisfying the following.
Let $X$ be a $n$-dimensional Fano variety. 
If $G$ is a finite group acting on $X$, then $G$ admits a normal abelian subgroup $A$ of index at most $i_0$ and rank at most $\dim X$. 
\end{theorem} 

The previous theorem is known as the {\em Jordan property} for finite actions on Fano varieties. 
In the 1870s, Camille Jordan proved that given a finite subgroup $G\leqslant {\rm GL}(n,\cc)$ there is a normal abelian subgroup $A\leqslant G$ of index at most $c(n)$. Note that, by simultaneous diagonalization over $\cc$, such abelian subgroup $A$ must have rank at most $n$. In~\cite{Col08}, Collins proved that we can take $c(n)=(n+1)!$ for $n\geq 71$. 
In the modern literature, statements asserting that finite subgroups of a geometric group $\mathcal{G}$ are almost abelian, in terms of some invariant of $\mathcal{G}$, are known as the Jordan property for $\mathcal{G}$.
The Jordan property is very common in algebraic geometry especially for finite actions on varieties.
We note that the previous theorem is also valid for rationally connected varieties, however, we mostly focus on the Fano setting throughout this article.

The previous theorem states that finite actions on $n$-dimensional Fano variety are {\em almost} abelian. Further, the rank of the acting abelian group is bounded above by the dimension of the variety. 
For instance, if $T$ is a toric Fano variety of dimension $n$, then we have $(\zz/m\zz)^n \leqslant {\rm Aut}(T)$ for every $m$.
One may wonder about the opposite direction: 
given an $n$-dimensional Fano variety $X$ with 
$(\zz/m\zz)^n \leqslant {\rm Aut}(X)$, where $m$ is large compared to $n$, what can we say about $X$? A priori, this question has nothing to do with cluster type varieties.
However, the following theorem shows a deep connection.
The proof of the following theorem is in~\cite[Theorem 2]{Mor20c}.

\begin{theorem}\label{thm:finite-actions}
There exists a constant $m_0:=m_0(n)$, only depending on $n$, satisfying the following. 
Let $X$ be a $n$-dimensional Fano type variety.
If $A:=(\zz/m\zz)^n \leqslant {\rm Aut}(X)$, with $m\geq m_0$, then $X$ is of cluster type. 
More precisely, there exists a torus embedding $\mathbb{G}_m^n\hookrightarrow X$ such that the volume form $\Omega_n$ has no zeros on $X$ and $A$ acts on the algebraic torus as a subgroup.
\end{theorem}

In other words, the $n$-dimesional Fano varieties with the largest finite actions must be of cluster type.
Explicit bounds for Theorem~\ref{thm:finite-actions} are not known. Some explicit values are expected to be found up to dimension three. One may wonder if $X$ in Theorem~\ref{thm:finite-actions} must be toric or not. The following example shows that cluster type is the best that one can hope for in this direction. 

\begin{example}
{\em 
Consider the projective space $\pp^2$ and the finite action of $A_m:=(\zz/m\zz)^2$ given by 
\[
[x:y:z]\mapsto [\mu_1x:\mu_2y:z]
\]
where $\mu_1$ and $\mu_2$ are two $m$-th roots of unity. 
Consider the set $Z_m:=\{ [0:\mu_2^k:1] \mid k\in \{1,\dots,m\}\}$. Then, the set $Z_m$ is invariant under the $A_m$-action on $\pp^2$.
Let $\pi_m \colon X_m \rightarrow \pp^2$ be the blow-up of $\pp^2$ along $X_m$. The variety $X_m$ is of Fano type 
and $A_m \leqslant {\rm Aut}(X_m)$. 
We argue that $X_m$ is not a toric variety, although it has a one-dimensional torus action. 
If $\mathbb{G}_m^2\leqslant {\rm Aut}(X_m)$, then the algebraic torus would fix all the $(-1)$-curves on $X_m$. Therefore, we would get a torus on $\mathbb{P}^2$ that fixes $Z_m$ pointwise. 
But, a $2$-dimensional torus action on $\pp^2$ fixes at most three points. On the other hand, we argue that $X_m$ is of cluster type.
Indeed, if we write 
\[
K_{X_m} + H_m = \pi_m^*(K_{\pp^2}+H_0+H_1+H_2) 
\]
then $H_m$ is just the strict transform of $H_0+H_1+H_2$ in $X_m$. 
Thus, the pair $(X_m,H_m)$ is Calabi--Yau of index one and we have an embedding $\mathbb{G}_m^2 \hookrightarrow X_m\setminus H_m$. 
Therefore, for $m$ sufficiently large, the sequence of surfaces $X_m$ admits $(\zz/m\zz)^2 \leqslant {\rm Aut}(X_m)$, and each $X_m$ is cluster type but not toric.
}
\end{example} 

In the setting of Theorem~\ref{thm:finite-actions}, we may drop the condition of $A$ having rank equal to the dimension of the underlying variety $X$. 
In this case~\cite[Theorem 1.1]{Mor24a} gives a reasonable geometric outcome whenever we have large finite actions on Fano type varieties (or more generally, rationally connected varieties). 

\begin{theorem}\label{thm:weak-geom-jordan}
There exists a constant $c(n)$, only depending on $n$, satisfying the following.
Let $X$ be a $n$-dimensional Fano type variety. 
Let $G$ be a finite group acting on $X$.
Then, there exists a normal abelian subgroup 
$A\leqslant G$ of index at most $c(n)$ and rank $k\leq n$ satisfying the following. 
There exists an $A$-equivariant birational map $X\dashrightarrow X'$ and an $A$-equivariant fibration $X'\rightarrow \pp^k$ such that $A$ induces an action on $\pp^k$ which factors through a maximal torus.
\end{theorem} 

The previous theorem can be regarded as a geometric version of the Jordan property for finite actions on Fano varieties. It says that large finite actions on Fano type varieties come from torus action on bases of fibrations. 
In the setting of Theorem~\ref{thm:weak-geom-jordan}, one may wonder if the torus action of $\pp^k$ can be lifted to a torus action on $X$. 
There are simple examples of birationally rigid therefolds $X_m$ for which $\zz/m\zz\leqslant {\rm Aut}(X_m)$, however, the variety $X_m$ does not admit an algebraic torus action, not even birationally.
In this direction, we propose a geometric Jordan property for the Cremona group.

\begin{conjecture}\label{conj:geom-jordan}
There exists a constant $c(n)$, only depending on $n$, satisfying the following. 
Let $X$ be a $n$-dimensional rational variety. 
Let $G$ be a finite group acting on $X$. 
Then, there exists a normal abelian subgroup $A\leqslant G$ of index at most $c(n)$ with 
$A\leqslant \mathbb{G}_m^k \leqslant {\rm Bir}(X)$.
\end{conjecture} 

Conjecture~\ref{conj:geom-jordan} would imply that, in the context of Theorem~\ref{thm:jordan-property}, the torus action of $\mathbb{G}_m^k$ in $\pp^k$ can indeed be lifted to some birational model of $X'$.
We may refer to Conjecture~\ref{conj:geom-jordan} as the {\em geometric Jordan property for the Cremona group}. This conjecture was proved in dimension two by Wang~\cite{Wan25}. The conjecture remains open in dimension three and higher.

\section{The cluster type condition in families} 
\label{sec:ct-in-families}

In this section, we discuss how the cluster type condition behaves in families of Fano varieties.
By the work of De Fernex and Fusi~\cite{dFF13}, in a family of varieties, the rational condition holds over a countable union of constructible sets. 
Furthermore, the locus of rational fibers in a family of smooth Fano varieties is not constructible by the work of Hassett, Pirutka, and Tschinkel~\cite{HPT18}.
In~\cite{JM24}, Ji and the author proved the following theorem. 

\begin{theorem}\label{thm:constructibility-cluster-type}
Let $\mathcal{X}\rightarrow T$ be a family of $\qq$-factorial terminal Fano varieties. Then, the set 
\[
T_{\rm cluster}:=\{ t\in T \mid \mathcal{X}_t \text{ is of cluster type}\} 
\]
is a constructible subset of $T$. 
\end{theorem} 

In particular, the previous theorem is valid for families of smooth Fano varieties. We expect the previous theorem to also be valid for families of klt Fano varieties, however, there are some technicalities in the proof related to invariance of plurigenera, which made us focus on the terminal case. Thus, the following problem comes along naturally. 

\begin{problem}
Prove that the cluster type condition is constructible in families of klt Fano varieties.
\end{problem} 

We argue that in families of varieties the cluster type condition is neither closed nor open. 

\begin{example}\label{ex:ct-not-closed}
{\em 
Consider a Gorenstein del Pezzo surface $\mathcal{X}_0$ of rank one with ${\rm vol}(\mathcal{X}_0)\geq 2$. 
Assume that $\mathcal{X}_0$ has $D$ singularities, 
so $\mathcal{X}_0$ is not of cluster type. 
For instance, we can choose $\mathcal{X}_0$ to have volume $4$ and a single $D_5$ singularity (see, e.g.,~\cite[Lemma 3]{MZ88}). 
As $-K_{\mathcal{X}_0}$ is big and $\mathcal{X}_0$ has only smoothable singularities, we conclude that $\mathcal{X}_0$ is indeed smoothable.
Let $\pi\colon \mathcal{X}\rightarrow \mathbb{D}$ be a smoothing of $\mathcal{X}_0$. 
Then, the variety $\mathcal{X}_t$ is a smooth Fano of degree $4$, for $t\neq 0$, and so $\mathcal{X}_t$ is of cluster type by Proposition~\ref{prop:Looijenga-pairs-are-cluster-type}.
Thus, the central fiber $\pi\colon \mathcal{X}\rightarrow \mathbb{D}$ is not cluster type, however, the general fiber of $\pi$ is cluster type.
A smoothing of a Gorenstein del Pezzo surface $X$ of rank one is cluster type provided $\vol(X)\geq 2$. 
This follows from Proposition~\ref{prop:nodal-curve}.
From this example, we conclude that the cluster type condition is not closed. 
}
\end{example} 

\begin{example}\label{ex:ct-not-open}
{\em 
Consider $\mathcal{Y}_0\subset \pp^4$ to be a Segre cubic. This variety is of cluster type. 
One can show this statement by describing 
$\mathcal{Y}_0$ as the anti-canonical model of a blow-up of $\pp^3$ at $5$ general points. 
We have a birational contraction $Z:={\rm Bl}_{p_1,\dots,p_5}(\pp^3)\rightarrow \mathcal{Y}_0$ that does not contract divisors.
Let $H_1,\dots,H_4$ be four hyperplanes such that each of the $5$ points $\{p_1,\dots,p_5\}$ is contained in at least two of the $H_i$'s. 
Let $\pi\colon Z\rightarrow \pp^3$ be the blow-up. 
Then, we obtain a Calabi--Yau pair 
\[
\pi^*(K_{\pp^3}+H_1+\dots+H_4)=K_Z+B_Z
\]
of index one and coregularity zero. 
Note that we have an embedding $\mathbb{G}_m^3\hookrightarrow Z\setminus B_Z$. 
So, the pair $(Z,B_Z)$ is of cluster type. 
If we push-forward $B_Z$ to $\mathcal{Y}_0$, then we obtain a boundary $B_{\mathcal{Y}_0}$ such that 
$(\mathcal{Y}_0,B_{\mathcal{Y}_0})$ is a cluster type pair. Let $\mathcal{Y}\rightarrow \mathbb{D}$ be a smoothing of the Segre cubic. 
Above, we concluded that $\mathcal{Y}_0$ is of cluster type. A smooth cubic $\mathcal{Y}_t$ is not rational, in particular, it cannot be cluster type. 
This example shows that the cluster type condition is not open.
}
\end{example} 

Example~\ref{ex:ct-not-closed} and Example~\ref{ex:ct-not-open} show that the cluster type condition is neither closed nor open. 
However, it is worth noticing that in both examples 
there is a presence of canonical singularities. 
It is still unclear whether the statement of Theorem~\ref{thm:constructibility-cluster-type} is optimal or not. Thus, we propose the following problem.

\begin{problem}
Prove that the cluster type condition is open in families of terminal Fano varieties. 
\end{problem} 

Together with the boundedness of smooth Fano varieties~\cite{KMM92b}, we conclude the following corollary. 

\begin{corollary}
Let $n$ be a positive integer. Then, there are only finitely many families of smooth Fano varieties of cluster type. 
\end{corollary}

The proof of Theorem~\ref{thm:constructibility-cluster-type} uses the following proposition which gives a characterization of cluster type pairs. 

\begin{proposition}\label{prop:MMP-ct}
Let $(X,B)$ be a log pair. 
The pair $(X,B)$ is of cluster type if and only if there exists a dlt modification $\phi\colon (Y,B_Y)\rightarrow (X,B)$ and a crepant birational contraction
$\psi\colon (Y,B_Y)\dashrightarrow (T,B_T)$ to a toric Calabi--Yau pair.
\end{proposition}

The birational contraction $\psi$ of Proposition~\ref{prop:MMP-ct} is indeed the run of an MMP for a suitable divisor on the variety $Y$. 
Proposition~\ref{prop:MMP-ct} is very useful to prove that a pair is cluster type. 
In practice, when we want to prove that a variety or pair is of cluster type, we study its dlt modifications and see whether we can explicitly produce an MMP that terminates in a toric variety.
Disproving that a variety is cluster type tends to be a harder problem (similar to rationality questions). 
In this case, Proposition~\ref{prop:MMP-ct} is not that useful as there exist countably many dlt modifications of a Calabi--Yau pair of coregularity zero. 

We briefly sketch the proof of Theorem~\ref{thm:constructibility-cluster-type}. 
Imagine that we have a family $\mathcal{X}\rightarrow T$ of smooth Fano varieties. 
We want to argue that if a general fiber $\mathcal{X}_t$ is of cluster type, then all fibers are cluster type over an open of the base. 
Using Chevalley's theorem on the image of constructible sets, we may take arbitrary finite covers of the base throughout our argument. 
First, we consider a cluster type boundary $\mathcal{B}_t \in |-K_{\mathcal{X}_t}|$ and, up to a finite cover of $T$, we extend it to nearby fibers. 
Thus, we have a family of Calabi--Yau pairs 
$(\mathcal{X},\mathcal{B})\rightarrow T$ where the underlying varieties are smooth Fano varieties. 
By Proposition~\ref{prop:MMP-ct}, there is a 
dlt modification $\phi_t \colon (\mathcal{Y}_t,\mathcal{B}_{\mathcal{Y}_t})\rightarrow (\mathcal{X}_t,\mathcal{B}_t)$ and a crepant birational contraction $\psi_t\colon (\mathcal{Y}_t,\mathcal{B}_{\mathcal{Y}_t})
\dashrightarrow (\mathcal{T}_t,\mathcal{B}_{\mathcal{T}_t})$ to a toric 
Calabi--Yau pair. 
The main aim of the proof is to argue that both $\phi_t$ and $\psi_t$ extend to the whole family $(\mathcal{X},\mathcal{B})\rightarrow T$.
This is achieved by extending divisors from $\mathcal{X}_t$ to the whole family $\mathcal{X}$. We also use some statements about the MMP in families. 
Finally, we need to argue that for the extended family 
$(\mathcal{T},\mathcal{B}_\mathcal{T})\rightarrow T$ 
all fibers near $t\in T$ are toric. This is done by computing the complexity of near fibers (see~\cite{JM24}).

In~\cite{MZ25}, the author and Z\'u\~niga proved some statements about limits of cluster type pairs.
The idea of these proofs closely follow the work of~\cite{JM24}.
One of the main theorems is the following. 

\begin{theorem}\label{thm:limits-ct}
Let $\pi\colon \mathcal{X}\rightarrow \mathbb{D}$ be a projective Fano fibration. Let $(\mathcal{X},\mathcal{X}_0+\mathcal{B})$ be a Calabi--Yau pair of index one over $\mathbb{D}$. 
Assume that $(\mathcal{X},\mathcal{X}_0)$ is plt and purely terminal on the complement of $\mathcal{B}$. 
If $(\mathcal{X}_t,\mathcal{B}_t)$ is of cluster type for $t\in \mathbb{D}^*$, then the central fiber 
$(\mathcal{X}_0,\mathcal{B}_0)$ is a finite quotient of a cluster type pair. 
\end{theorem} 

Theorem~\ref{thm:limits-ct} has some technical conditions. However, all conditions are satisfied if we assume that $\mathcal{X}\rightarrow \mathbb{D}$ is a family of terminal Fano varieties. The previous theorem can be thought of as saying that limits of cluster type pairs are finite quotients of cluster type pairs. 
A similar statement for toric varieties is proved in~\cite[Theorem 1.1]{MZ25}. We conclude this section by proposing the following problem, which seems to be challenging:

\begin{problem}\label{prob:toric-deg}
Prove that every cluster type pair admits a toric degeneration. 
\end{problem} 

In the previous problem, we mostly care about normal toric degenerations. In the setting of cluster varieties, i.e., the spectrum of cluster algebras and suitable compactifications, there are plenty of results in the literature proving that such geometric objects admit many toric degenerations (see, e.g.,~\cite{BLMM17,BFMMNC20}). However, in the cluster type setting, there is no proof yet that this is always the case. 
Providing a positive answer for the previous problem would be a quintessential piece towards the classification of cluster type varieties. 

\section{Cox rings of cluster type varieties}
\label{sec:ct-and-cox-rings}

In this section, we discuss Cox rings of cluster type varieties. To understand the main theorem in this direction, we need to introduce the concept of {\em mutation semigroup algebras}.
Throughout this section, $N$ denotes a free, finitely generated abelian group, and $M$ denotes its dual. 

\begin{definition} 
{\em 
An {\em algebraic mutation datum}
is a pair $(u,h=g^k)$, where $u\in N$
is a primitive vector, 
$g \in \cc[u^\perp\cap M]$ is an irreducible Laurent polynomial, and $k$ is a positive integer. 
For a polyhedral cone $\sigma \subset N$, 
we say that $(u,h=g^k)$ is {\em $\sigma$-admissible} if $u\not \in \sigma$.

Let $\sigma \subset N_\qq$ be a rational polyhedral cone and $\cc \hookrightarrow F$ be a field extension. 
A $\cc$-algebra homomorphism 
$\iota\colon \cc[\sigma^\vee\cap M]\hookrightarrow F$ is an {\em embedded semigroup algebra} if it induces an isomorphism $\cc(M)\simeq F$.
}
\end{definition} 

\begin{definition}
{\em 
Let $\sigma_1,\sigma_2\subset N_\qq$ be two rational polyhedral cones and $U_{\sigma_1},U_{\sigma_2}$ be the corresponding affine toric varieties. 
A birational map $\mu\colon U_{\sigma_1}\dashrightarrow U_{\sigma_2}$ is a {\em mutation} if the two following conditions are satisfied:
\begin{enumerate}
\item the induced isomorphism $\mu^*\colon \cc(M)\rightarrow \cc(M)$ is given on monomials by $\mu^*(x^m)=x^mh^{-\langle u,m\rangle}$ for some $\sigma_2$-admissible algebraic mutation datum $(u,h)$, and 
\item strict transform via $\mu$ induce a bijection between the prime torus invariant divisors of $U_{\sigma_1}$ and $U_{\sigma_2}$.
\end{enumerate} 
Two embedded semigroup algebras 
$\iota_i \colon \cc[\sigma_1^\vee\cap M]\hookrightarrow F$ with $i\in \{1,2\}$ 
are said to {\em differ by a mutation} if 
$\iota_1^{-1}\circ \iota_2$ is a mutation.
}
\end{definition} 

\begin{definition}
{\em 
A {\em mutation semigroup algebra} is a finitely generated ring $R$ over $\cc$ which can be expressed as 
\[
R=R_0\cap \dots \cap R_k 
\]
where the following conditions hold for $i\in \{1,\dots,k\}$:
\begin{enumerate}
\item there is an embedded semigroup algebra
$\iota_i \colon \cc[\sigma_i^\vee\cap M]\hookrightarrow {\rm Frac}(R)$ with image
$R_i$;
\item the embedded semigroup algebras $\iota_0$ and $\iota_i$ differ by a mutation; and
\item for each prime ideal $\mathfrak{p}\subset R_i$ of height one, the prime ideal $\mathfrak{p}\cap R_i$ has height one in $R_i$.
\end{enumerate} 
We abbreviate mutation semigroup algebra by MSA.
We say that $R$ is just a {\em mutation algebra} if each $\sigma_i=\{0\}$.
}
\end{definition}

In~\cite{EFMS25}, we proved that cluster type pairs and mutation semigroup algebras are closely related. More precisely, we proved the following theorem.

\begin{theorem}\label{thm:msa-vs-ct}
Let $R$ be a finitely generated $\cc$-algebra.
Assume that $U={\rm Spec}(R)$ has klt singularities. 
Then, the following two statements are equivalent:
\begin{enumerate}
\item $R$ is a mutation semigroup algebra; and 
\item there exists a cluster type pair $(X,B)$, an ample divisor $0\leqslant A \leqslant B$ and an isomorphism 
\[
\mathcal{O}(X\setminus A)\simeq R.
\]
\end{enumerate}
\end{theorem} 
The spectrum of mutation semigroup algebras can be regarded as natural affine charts of cluster type varieties. They play the same role as affine toric varieties in projective toric varieties.

Let $X$ be a normal projective variety with 
free finitely generated Class group ${\rm Cl}(X)$. Let $D_1,\dots,D_k$ be Weil divisors on $X$ that generate the Class group. 
The {\em Cox ring} of $X$ is defined to be: 
\[
{\rm Cox}(X):=\bigoplus_{(m_1,\dots,m_k)\in \zz^k} H^0(X,\mathcal{O}_X(m_1D_1+\dots+D_km_k))
\]
where the multiplication happens in the function field of $X$. We say that a normal projective variety is a {\em Mori dream space} if its Cox ring is finitely generated over $\cc$. Cox rings and Mori deam spaces can be defined more generally whenever ${\rm Cl}(X)$ is finitely generated (see, e.g.,~\cite[Definition 3.4]{BM21}). 
Toric varieties can be characterized as the Mori deam spaces with Cox rings being polynomial rings~\cite[Corollary 2.10]{HK00}.
Fano type varieties can be characterized as the Mori dream spaces with Cox rings having Gorenstein canonical singularities (see, e.g.,~\cite{GOST15,Bra19}). Using Theorem~\ref{thm:msa-vs-ct}, we can show that the cluster type condition can be detected at the level of Cox rings as well.

\begin{theorem}\label{thm:ct-cox-rings}
Let $X$ be a klt Fano variety. 
The variety $X$ is of cluster type
if and only if ${\rm Cox}(X)$ is a ${\rm Cl}(X)$-graded mutation semigroup algebra.
\end{theorem} 

The previous theorem provides a new algebraic tool for proving that a klt Fano variety is of cluster type. Theorem~\ref{thm:ct-cox-rings} suggests that many smooth Fano threefolds are cluster type.
Indeed, in~\cite{DHHKL15}, Derenthal, Hausen, Heim, Keicher, and Laface compute the Cox rings of many smooth Fano threefolds, and they resemble mutation semigroup algebras (see~\cite[Theorem 4.5]{DHHKL15}).
In~\cite[Corollary 1.8]{EFMS25}, we also prove that many varieties related to Lie theory are Fano type and cluster type.
More precisely, we prove that Richardson varieties, flag varieties, Schubert varieties, and Bott-Samelson varieties are Fano type and cluster type.

\bibliographystyle{habbvr}
\bibliography{bib}

\vspace{0.5cm}
\end{document}